\documentclass{article}%
\usepackage{amsmath}%
\usepackage{amsfonts}%
\usepackage{amssymb}%
\usepackage{graphicx}
\usepackage{hyperref}
\usepackage{moresize}

\usepackage{epstopdf} 
\newtheorem{theorem}{Theorem}

\newtheorem{definition}[theorem]{Definition}
\newtheorem{example}[theorem]{Example}

\newtheorem{lemma}[theorem]{Lemma}

\newtheorem{remark}[theorem]{Remark}

\newenvironment{proof}[1][Proof]{\noindent\textbf{#1.} }{\ \rule{0.5em}{0.5em}}

\title{Geometric Proof of Strong Stable/Unstable Manifolds, with Application to the Restricted Three Body Problem}

\author{Maciej J. Capi\'nski \thanks{ Faculty of Applied Mathematics, AGH University of Science and Technology, al. Mickiewicza 30, 30-059
Krak\'ow,  Poland ({\tt mcapinsk@agh.edu.pl})  supported by Polish Nacional Science Center Grants NCN 2012/05/B/ST1/00355 and N201 543238, together with the United States NSF Grant DMS1162544} \and 
Anna Wasieczko \thanks{Faculty of Applied Mathematics, AGH University of Science and Technology, al. Mickiewicza 30, 30-059
Krak\'ow,  Poland ({\tt wasieczk@agh.edu.pl})}}
\begin{document}
\maketitle
\begin{abstract}
We present a method for establishing strong stable/unstable manifolds of fixed points for maps and ODEs. The method is based on cone conditions, suitably formulated to allow for application in computer assisted proofs. In the case of ODEs, assumptions follow from estimates on the vector field, and it is not necessary to integrate the system.
We apply our method to the restricted three body problem and show that for a given choice of the mass parameter, there exists a homoclinic orbit along matching strong stable/unstable manifolds of one of the libration points. 
\end{abstract}

\section{Introduction}

In this paper we give a geometric method for establishing strong
stable/unstable invariant manifolds of fixed points. The method is based on a
graph transform type approach. Its assumptions are founded on suitably defined
cone conditions, which can be verified using rigorous (interval arithmetic
based), computer--assisted numerics.

Our approach is in a similar spirit to a number of previous results. The
papers \cite{GZ1,GZ2} by Gidea and Zgliczy\'{n}ski introduced a topological
tool referred to as \textquotedblleft covering relations\textquotedblright\ or
\textquotedblleft correctly aligned windows\textquotedblright. The tool can be
applied to obtain computer assisted proofs of symbolic dynamics in dynamical
systems. A paper \cite{Z} by Zgliczy\'{n}ski extends the method by adding
suitable cone conditions. With such additional assumptions one can establish
existence of hyperbolic fixed points and their associated stable and unstable
manifolds. The method has also been adapted by Zgliczy\'{n}ski, Sim\'{o} and
Capi\'{n}ski for proofs of normally hyperbolic invariant manifolds
\cite{C,CS,CZ-cc}. The above methods have been used and applied to a number of
systems including the restricted three body problem \cite{C2, CR,
Zgliczyn-Wilczak, Zgliczyn-Wilczak2} rotating H\'{e}non map \cite{C, CZ-cc},
driven logistic map \cite{CS}, forced damped pendulum \cite{WZ-pendulum}, and
proofs of slow manifolds \cite{GJM}. All these results rely on
suitable definitions of covering relations and cone conditions. The result
presented in this paper deals with fixed points, and is closely related to
\cite{Z}. The main difference is that our result can be used to establish
strong (un)stable manifolds, which could be submanifolds of the full
(un)stable manifold. Our method can also be applied to saddle--center fixed
points, which is not possible using \cite{Z}, since it relies on
hyperbolicity. Finally, our method does not rely on covering relations, which
reduces the number of assumptions by half and simplifies their verification.

There are a number of alternative approaches for computer assisted proofs of
invariant manifolds. These involve solving an appropriate fixed point equation
in a functional setting. Amongst these methods it is notable to mention the
work of Cabre, de la Llave, and Fontich \cite{Llave1,Llave2,Llave}. Our approach is different. It follows from a
topological argument performed in the state space of the system, instead of
considering the problem in a functional setting. The assumptions of our
theorem are simpler to verify, but at the cost of obtaining less accurate
bounds on the manifold enclosure.

As an example of an application of our method we consider the planar circular
restricted three body problem. We use the method to establish a rigorous
enclosure of an unstable manifold of a libration fixed point of the problem.
Using continuity based arguments, we also prove that the fixed point has a
homoclinic orbit, for a suitably chosen parameter of the system. The example
considered by us has first appeared in the work by Llibre, Martinez and
Sim\'{o} \cite{Simo}, where existence of such homoclinic connections has been
demonstrated numerically. We validate their results using rigorous, interval
based, computer assisted numerics.

To the best of our knowledge, our result is amongst the first computer
assisted proofs of nontransversal homoclinic orbits for ODEs. The only other
result known to us is the work of Szczelina and Zgliczy\'nski \cite{SZ}, where
a homoclinic orbit is proved for a two dimensional ODE. We note that the
considered by us homoclinic connection in the restricted three body problem
has not been proved up till this point. The only proof is the result of
Llibre, Martinez and Sim\'{o} \cite{Simo}, where an analytic argument is given
for a sufficiently small mass parameter. Their method can not be applied though
for a concrete given parameter, which is what we do in this paper.

Establishing of homoclinic connections between invariant objects can be used
in the study of stability of a system. Combined with Melnikov type arguments,
these can be used in proofs of Arnold diffusion \cite{Arnold} type dynamics. A
broad selection of papers has used this approach, including the work of
Delshams, Huguet, de la Llave, Seara or Treschev
\cite{DH1,DH2,DLS1,DLS2,T1,T2} amongst many others. Such approach has
also been applied in \cite{CZ}, in the setting of the planar elliptic
restricted three body problem. It used the homoclinic connections from
\cite{Simo} for the Melnikov method. The result though was not fully rigorous,
and relied on numerical computation of Melnikov integrals. The rigorous
enclosure of the homoclinic orbit established in this paper can be a starting
point for a rigorous validation of this computation. This would lead to a
proof of Arnold diffusion type dynamics in the elliptic restricted three body
problem. We plan to perform such validation in forthcoming work.

The paper is organized as follows. Section \ref{sec:prelim} contains
preliminaries. In section \ref{sec:statement} we state our results. In section
\ref{sec:cones-discs} we present auxiliary results concerning cone conditions,
which are then used in the proofs of our main results in section
\ref{sec:construction}. Our results are written for maps. In section
\ref{sec:flows} we show how they can be applied for ODEs. Section
\ref{sec:3bp-application} contains an application of our method, and contains
a proof of existence of a homoclinic orbit to the libration point $L_{1}$ in
the restricted three body problem. Sections \ref{sec:closing-remarks},
\ref{sec:acknowledgements} and \ref{sec:appendix} contain closing remarks,
acknowledgements and the appendix, respectively.


\section{Preliminaries\label{sec:prelim}}

\subsection{Notations}

Throughout the paper, all norms that appear are standard Euclidean norms. We
use a notation $B(x,r)$ to denote a ball of radius $r$ centered at $x$. If we
want to emphasize that a ball is in $\mathbb{R}^{k}$, then we add a subscript
and write $B_{k}(x,r)$. We use a short hand notation $B_{k}=B_{k}(0,1)$ for a
unit ball in $\mathbb{R}^{k}$ centered at zero. For a set $A\subset
\mathbb{R}^{k}$ we use $\overline{A}$ to denote its closure and $\partial A$
for its boundary. For a point $p=\left(  x,y\right)  $ we use a notation
$\pi_{x}p$ and $\pi_{y}p$ to denote projections onto $x$ and $y$ coordinates, respectively.

\subsection{Computer assisted proofs}

Most computations performed on a computer are burdened with error. Even very
simple operations on real numbers (such as adding, multiplying or dividing)
can result in round off errors. To make computer assisted computations fully
rigorous, one can employ interval arithmetic, where instead of real numbers
one deals with intervals. Any operation is made rigorous by appropriate
rounding, which ensures an enclosure of the true result.

Interval arithmetic can also be used to treat basic functions (such as $\sin$,
$\cos$ or exponent). It can be extended to perform linear algebra on interval
vectors and interval matrices. One can thus design algorithms which give
rigorous enclosures for multiplying matrices, inverting a matrix, computing
eigenvectors or solving linear equations.

The interval arithmetic approach can also be extended to treat functions
$f:\mathbb{R}^{n}\to\mathbb{R}^{m}$. One can implement algorithms which
compute interval enclosures for images of the function $f$, for its derivative
and for higher order derivatives.

The interval arithmetic approach can also be used for the integration of ODEs.
One can implement interval arithmetic based integrators, which allow for the
computation of enclosures of the images of points along of a flow of an ODE.
One can extend such integrators to include the computation of high order
derivatives of a time shift map along the flow, or even to compute high order
derivatives for Poincar\'e maps \cite{Cn-Lohner}.

All above mentioned tasks can be performed using a single \texttt{C++} library
``Computer Assisted Proofs in Dynamics" (CAPD for short). The package is
freely available at:\smallskip

\href{http://capd.ii.uj.edu.pl}{http://capd.ii.uj.edu.pl} \smallskip

\noindent All the computer assisted proofs from this paper have been performed
using CAPD.

\subsection{Interval Newton Method}

Let $X$ be a subset of $\mathbb{R}^{n}$. We shall denote by $[X]$ an interval
enclosure of the set $X$, that is, a set
\[
[X]=\Pi_{i=1}^{n} [a_{i},b_{i}]\subset\mathbb{R}^{n},
\]
such that
\[
X\subset[X].
\]

Let $f:\mathbb{R}^{n}\to\mathbb{R}^{n}$ be a $C^{1}$ function and
$U\subset\mathbb{R}^{n}$. We shall denote by $[Df(U)]$ the interval enclosure
of a Jacobian matrix on the set $U$. This means that $[Df(U)]$ is an interval
matrix defined as
\[
\lbrack Df(U)]=\left\{  A\in\mathbb{R}^{n\times n}|A_{ij}\in\left[  \inf_{x\in
U}\frac{df_{i}}{dx_{j}}(x),\sup_{x\in U}\frac{df_{i}}{dx_{j}}(x)\right]
\text{ for all }i,j=1,\ldots,n\text{ }\right\}  .
\]

\begin{theorem}
\cite{Al} (\textbf{Interval Newton method}) \label{th:interval-Newton} Let
$f:\mathbb{R}^{n}\rightarrow\mathbb{R}^{n}$ be a $C^{1}$ function and
$X=\Pi_{i=1}^{n}[a_{i},b_{i}]$ with $a_{i}<b_{i}$. If $[Df(X)]$ is invertible
and there exists an $x_{0}$ in $X$ such that%
\[
N(x_{0},X):=x_{0}-\left[  Df(X)\right]  ^{-1}f(x_{0})\subset X,
\]
then there exists a unique point $x^{\ast}\in X$ such that $f(x^{\ast})=0.$
\end{theorem}

\subsection{Restricted three body problem}

\label{sec:RTBP}

The problem is defined as follows: two main bodies rotate in the plane about
their common center of mass on circular orbits under mutual gravitational
influence. A third body moves in the same plane of motion as the two main
bodies, attracted by their gravitation, but not influencing their motion. The
problem is to describe the motion of the third body.

Usually, the two rotating bodies are referred to as the \emph{primaries}. The
third body can be regarded as a satellite or a spaceship of negligible mass.

We use a rotating system of coordinates centred at the center of mass. The
plane $X,Y$ rotates with the primaries. The primaries are on the $X$ axis, the
$Y$ axis is perpendicular to the $X$ axis and contained in the plane of rotation.

\begin{figure}[ptb]
\begin{center}
\includegraphics[height=4cm]{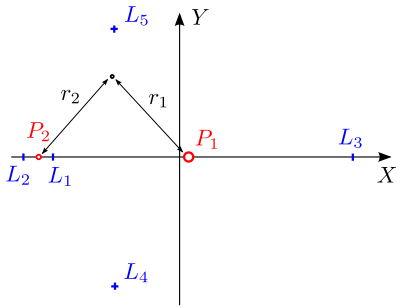}
\end{center}
\caption{Rotating system of coordinates with origin at the center of mass. The
sun has the mass $1-\mu$ and is fixed at $P_{1}=(\mu,0)$. The planet has the
mass $\mu$ is fixed at $P_{2}=(\mu-1,0)$. The third massless particle moves in
the $XY$ plane.}%
\label{fig:3bp}%
\end{figure}We rescale the masses $\mu_{1}$ and $\mu_{2}$ of the primaries so
that they satisfy the relation $\mu_{1}+\mu_{2}=1$. After such rescaling the
distance between the primaries is $1$. (See Szebehelly~\cite{Szeb}, section
1.5). We refer to the larger of the two primaries as the ``sun" and to the
smaller as the ``planet". We use a convention in which in the rotating
coordinates the sun is located to the right of the origin at $P_{1}=(\mu,0)$,
and the planet is located to the left at $P_{2}=(\mu-1,0)$.

The equations of motion of the third body are
\begin{subequations}
\label{eq:RTBP}%
\begin{align*}
\ddot{X}-2\dot{Y}  &  =\Omega_{X},\\
\ddot{Y}+2\dot{X}  &  =\Omega_{Y},
\end{align*}
where
\end{subequations}
\[
\Omega=\frac{1}{2}(X^{2}+Y^{2})+\frac{1-\mu}{r_{1}}+\frac{\mu}{r_{2}},
\]
and $r_{1},r_{2}$ denote the distances from the third body to the larger and
the smaller primary, respectively (see Figure \ref{fig:3bp})
\begin{align*}
r_{1}^{2}  &  =(X-\mu)^{2}+Y^{2},\\
r_{2}^{2}  &  =(X-\mu+1)^{2}+Y^{2}.
\end{align*}
These equations have an integral of motion~\cite{Szeb} called the Jacobi
integral
\[
C=2\Omega-(\dot{X}^{2}+\dot{Y}^{2}).
\]

The equations of motion take Hamiltonian form if we consider positions $X$,
$Y$ and momenta $P_{X}=\dot{X}-Y$, $P_{Y}=\dot{Y}+X$. The Hamiltonian is
\begin{equation}
H=\frac{1}{2}(P_{X}^{2}+P_{Y}^{2})+YP_{X}-XP_{Y}-\frac{1-\mu}{r_{1}}-\frac
{\mu}{r_{2}}, \label{eq:Hamiltonian}%
\end{equation}
with the vector field given by%
\begin{align*}
F  &  =J\nabla H,\\
J  &  =\left(
\begin{array}
[c]{cc}%
0 & \mathrm{id}\\
-\mathrm{id} & 0
\end{array}
\right)  ,\quad\mathrm{id}=\left(
\begin{array}
[c]{cc}%
1 & 0\\
0 & 1
\end{array}
\right)  .
\end{align*}
The Hamiltonian and the Jacobi integral are simply related by $H=-\frac{C}{2}$.

Due to the Hamiltonian integral, the dimensionality of the space can be
reduced by one. Trajectories of the system stay on the \emph{energy surface}
$M$ given by $H(X,Y,P_{X},P_{Y})=h=$\textit{constant}. Equivalently, $M$ is
the level surface
\begin{equation}
M\equiv\{C(X,Y,\dot{X},\dot{Y})=c=-2h\} \label{eq:energySurface}%
\end{equation}
of the Jacobi integral.

The problem has a reversing symmetry defined by
\begin{equation}
S(X,Y,P_{X},P_{Y})=\left(  X,-Y,-P_{X},P_{Y}\right)  . \label{eq:S-sym}%
\end{equation}
Using a notation $\mathbf{x}=\left(  X,-Y,-P_{X},P_{Y}\right)  $ for the
coordinates, and $\phi_{t}(\mathbf{x})$ for the flow of the vector field%
\[
\mathbf{\dot{x}}=J\nabla H(\mathbf{x}),
\]
the system has the property
\begin{equation}
S(\phi_{t}(\mathbf{x}))=\phi_{-t}(S(\mathbf{x})). \label{eq:S-sym-property}%
\end{equation}

The problem has five equilibrium points (see \cite{Szeb}). Three of them,
denoted $L_{1},L_{2}$ and $L_{3}$, lie on the $X$-axis and are usually called
the `collinear' equilibrium points (see Figure \ref{fig:3bp}). Notice that we
denote $L_{1}$ the \emph{interior} collinear point, located between the primaries.

The Jacobian of the vector field at $L_{1}$ has two real and two purely
imaginary eigenvalues. It possesses a one dimensional unstable manifold. By
(\ref{eq:S-sym-property}), the one dimensional stable manifold is
$S$-symmetric to the unstable manifold.

\section{Statement of main results\label{sec:statement}}

Our paper contains two results. The first is a method for establishing strong
invariant manifolds for fixed points. The method is based on cone conditions,
and is tailor made for rigorous (interval based) computer assisted
implementation. The second result is an application of the method to prove a
homoclinic connection of a libration fixed point in the restricted three body problem.

\subsection{Establishing strong invariant
manifolds\label{sec:main-results-manifolds}}

Let $N=\overline{B}_{u}\times\overline{B}_{s}$ and
\[
f:N\rightarrow\mathbb{R}^{u}\times\mathbb{R}^{s}%
\]
be a $C^{1}$ function. We assume that there exists a fixed point for $f$ in
the interior of $N$. For simplicity we assume that the fixed point is at zero.
This assumption can easily be relaxed (see Remark \ref{rem:fixed-point}).

Our method can be applied to establish strong stable and strong unstable
manifolds defined as follows:

\begin{definition}
Let $U$ be a neighborhood of zero and let $\lambda<1$. A set $W_{\lambda
,U}^{s}\subset U$ consisting of all points $p\in U$ satisfying:

\begin{enumerate}
\item $f^{n}(p)\in U$ for any $n\in\mathbb{N}$;

\item there exists a constant $C>0$ (which can depend on $p$), such that for
all $n \ge0$%
\begin{equation}
\left\Vert f^{n}(p)\right\Vert \leq C\lambda^{n}; \label{eq:contr-rate-cond}%
\end{equation}

\end{enumerate}

\noindent is called a strong stable manifold, with contraction rate $\lambda$,
in $U$.
\end{definition}

\begin{definition}
Let $U\subset\mathbb{R}^{u}\times\mathbb{R}^{s}$ be a set and let $p\in U$. We
say that a sequence $(p_{0},p_{-1},p_{-2},\ldots)$ is a backward trajectory of
$p$ in $U$ if $p_{0}=p$ and for any $i<0$, $p_{i}\in U$ and $p_{i+1}=f(p_{i})$.
\end{definition}

\begin{definition}
Let $U$ be a neighborhood of zero and let $\lambda>1$. A set $W_{\lambda
,U}^{u}\subset U$ consisting of all points $p\in U$ satisfying:

\begin{enumerate}
\item there exists a backward trajectory $(p_{0},p_{-1},p_{-2},\ldots)$ of $p$
in $U$;

\item for any backward trajectory $(p_{0},p_{-1},p_{-2},\ldots)$ of $p$ in $U$
there exists a constant $C>0$ (which can depend on the backward trajectory),
such that for all $n\leq0$%
\begin{equation}
\left\Vert p_{n}\right\Vert \leq C\lambda^{n}; \label{eq:exp-rate-cond}%
\end{equation}

\end{enumerate}

\noindent is called a strong unstable manifold, with expansion rate $\lambda$,
in $U$.
\end{definition}

\begin{example}
\label{exmpl:contr-exp-rate}Let $f_{1}(x,y)=(\frac{1}{2}x,\frac{1}{3}y)$. The
stable manifold with contraction rate $\frac{1}{2}$ in $\mathbb{R}^{2}$ is
equal to $\mathbb{R}^{2}$ and the stable manifold with contraction rate
$\frac{1}{3}$ in $\mathbb{R}^{2}$ is $\left\{  0\right\}  \times\mathbb{R}$.
Similarly, for $f_{2}(x,y)=(2x,3y)$ the unstable manifold with expansion rate
$2$ in $\mathbb{R}^{2}$ is $\mathbb{R}^{2}$ and the unstable manifold with
expansion rate $3$ in $\mathbb{R}^{2}$ is $\left\{  0\right\}  \times
\mathbb{R}$.
\end{example}

Let $\alpha_{h},\alpha_{v}\in(0,1)$ and let
\[
Q_{h},Q_{v}:\mathbb{R}^{u}\times\mathbb{R}^{s}\rightarrow\mathbb{R}%
\]
be defined as%
\begin{align}
Q_{h}(x,y)  &  =\alpha_{h}\left\Vert x\right\Vert ^{2}-\left\Vert y\right\Vert
^{2}\label{eq:cones-def1}\\
Q_{v}(x,y)  &  =\left\Vert x\right\Vert ^{2}-\alpha_{v}\left\Vert y\right\Vert
^{2}. \label{eq:cones-def2}%
\end{align}

\begin{definition}
\label{def:f_cone_cond} Let $\alpha,\beta,m>0$ and let $Q(x,y)=\alpha
\left\Vert x\right\Vert ^{2}-\beta\left\Vert y\right\Vert ^{2}.$ We say that
$f$ satisfies cone conditions for $(Q,m)$ in $N$ if for any $p_{1}\neq p_{2},$
$p_{1},p_{2}\in N,$ holds
\[
Q(f(p_{1})-f(p_{2}))>mQ(p_{1}-p_{2}).
\]

\end{definition}

The following theorems are the main results of our paper.

\begin{theorem}
\label{th:main-unstable}Assume that $m_{v}>m_{h}>0$ and $m_{v}>1$. Let
$r^{u}=\sqrt{1-\alpha_{v}}$ and $U=\overline{B}_{u}(0,r^{u})\times\overline
{B}_{s}$. If $f$ satisfies cone conditions for $(Q_{h},m_{h})$ and
$(Q_{v},m_{v})$ in $N,$ then there exists a function $w^{u}:\overline{B}%
_{u}(0,r^{u})\rightarrow\overline{B}_{s},$ such that
\[
W_{\sqrt{m_{v}},N}^{u}\cap U=\left\{  (x,w^{u}(x))|x\in\overline{B}%
_{u}(0,r^{u})\right\}  .
\]
Moreover, $w^{u}$ is Lipschitz with a constant $L_{u}=\sqrt{\alpha_{h}}$.
\end{theorem}

\begin{theorem}
\label{th:main-stable}Assume that $m_{v}>m_{h}>0$ and $m_{h}<1.$ Let
$r^{s}=\sqrt{1-\alpha_{h}}$ and $U=\overline{B}_{u}\times\overline{B}%
_{s}(0,r^{s}).$ If $f$ satisfies cone conditions for $(Q_{h},m_{h})$ and
$(Q_{v},m_{v})$ in $N,$ then there exists a function $w^{s}:\overline{B}%
_{s}(0,r^{s})\rightarrow B_{u},$ such that
\[
W_{\sqrt{m_{h}},N}^{s}\cap U=\left\{  (w^{s}(y),y)|y\in\overline{B}%
_{s}(0,r^{s})\right\}  .
\]
Moreover, $w^{s}$ is Lipschitz with a constant $L_{s}=\sqrt{\alpha_{v}}.$
\end{theorem}

The proofs of Theorems \ref{th:main-unstable}, \ref{th:main-stable} are given
in Section \ref{sec:construction}.

\begin{remark}
Let us say that the fixed point has a stable manifold. Theorem
\ref{th:main-stable} can be used to establish a lower dimensional manifold
(which is a sub manifold of the full stable manifold), that is associated with
some prescribed contraction rate. For instance, $f_{1}$ from Example
\ref{exmpl:contr-exp-rate} has such a lower dimensional stable manifold that
is associated with contraction rate $\lambda=\frac{1}{3}$.

Similarly, Theorem \ref{th:main-unstable} can be used to establish lower
dimensional submanifolds of an unstable manifold, that are associated with
prescribed expansion rates. The $f_{2}$ from Example
\ref{exmpl:contr-exp-rate} has such a lower dimensional unstable manifold that
is associated with expansion rate $\lambda=3$.
\end{remark}

Theorems \ref{th:main-unstable}, \ref{th:main-stable} are formulated for maps.
In Section \ref{sec:flows} we show mirror results for flows (see Theorems
\ref{th:odes-wu}, \ref{th:odes-ws}). We emphasize that these results do not
require rigorous integration, but follows directly from appropriate bounds on
the vector field.

Let us point out that assumptions of Theorems \ref{th:main-unstable},
\ref{th:main-stable} can easily be verified using the following lemma.

\begin{lemma}
\label{lem:cc-pos-def}Let $\alpha,\beta,m>0$ and let $Q(x,y)=\alpha\left\Vert
x\right\Vert ^{2}-\beta\left\Vert y\right\Vert ^{2}.$ Assume that for any
$B\in\left[  Df(N)\right]  ,$ the quadratic form%
\[
V(q)=Q(Bq)-mQ(q)
\]
is positive definite, then $f$ satisfies cone conditions for $(Q,m)$ in $N$.
\end{lemma}

\begin{proof}
The proof is given in Appendix \ref{app:cc-pos-def}.
\end{proof}

There are a number of algorithms that can be used to verify if a matrix is
positive definite. This can also be done using interval arithmetic.

Let us finish the section with simple examples, which provide some intuition
for the results.

\begin{example}
Let $a,b\in\mathbb{R}$, $a >b>0 $, and let $f:N\rightarrow\mathbb{R}^{2}$ be a
linear map
\[
f(x,y)=\left(  ax,by\right)  ,
\]
then $f$ satisfies cone conditions for $(Q_{h},m_{h})$ and $(Q_{v},m_{v})$ for
any $m_{h},m_{v}\in\left(  b^{2},a^{2}\right)  $. Let us note that we do not
need to assume that $a >1$ or that $b <1$. Moreover, $f_{\varepsilon
}=f+\varepsilon g$ satisfies cone conditions, provided that $g$ is
differentiable and $\varepsilon$ is small enough.
\end{example}

\begin{example}
Let $a,b\in\mathbb{R}$, $a>1> b>0$ and let $R:\mathbb{R}^{2}\ni\theta
\rightarrow R(\theta)\in\mathbb{R}^{2}$ be a rotation. Consider $f:\mathbb{R}%
^{4}\rightarrow\mathbb{R}^{4},$ of the form%
\[
f\left(  \xi,\theta,\eta\right)  =\left(  a\xi,R(\theta),b\eta\right)  .
\]

For coordinates $x_{1},y_{1}$ chosen as
\[
x_{1}=\xi,\qquad y_{1}=(\theta,\eta),
\]
assumptions of Theorem \ref{th:main-unstable} are satisfied for any
$\alpha_{h},\alpha_{v}\in\left(  0,1\right)  $ and any $m_{h},m_{v}\in
(1,a^{2})$ satisfying $m_{v}>m_{h}$.

On the other hand, for coordinates $x_{2},y_{2}$ chosen as
\[
x_{2}=\left(  \xi,\theta\right)  ,\qquad y_{2}=\eta,
\]
assumptions of Theorem \ref{th:main-stable} are satisfied for any $\alpha
_{h},\alpha_{v}\in\left(  0,1\right)  $ and any $m_{h},m_{v}\in(b^{2},1)$
satisfying $m_{v}>m_{h}$

We thus see that we can apply Theorems \ref{th:main-unstable} and
\ref{th:main-stable} by swapping the roles of some of the coordinates.

The assumptions still hold for $f_{\varepsilon}=f+\varepsilon g,$ provided
that $g$ is differentiable and $\varepsilon$ is small enough.
\end{example}

We conclude this section with a remark that the fixed point does not need to
be centered at zero in order to apply our method.

\begin{remark}
\label{rem:fixed-point}The proofs of Theorems \ref{th:main-unstable} and
\ref{th:main-stable} are conducted under the assumption that the fixed point
is at zero. In many applications though it can be difficult to establish the
fixed point analytically. In computer assisted proofs the enclosure of a fixed
point can be obtained using the interval Newton theorem (Theorem
\ref{th:interval-Newton}). Assuming that we know that the fixed point is
contained in a set $B\subset\mathbb{R}^{u+s},$ it is sufficient to verify cone
conditions on a set $N^{\prime}=N+B$.
\end{remark}

\subsection{Establishing existence of homoclinic orbits in the restricted
three body problem}

In the work by Llibre, Martinez and Sim\'o \cite{Simo} it is shown that for
suitably chosen family of parameters $\mu\in\{\mu_{k}^{\ast}\}_{k=2}^{\infty}%
$, $\mu_{k+1}^{\ast}<\mu_{k}^{\ast},$ the unstable and stable manifolds of
$L_{1}^{\mu_{k}^{\ast}}$ coincide, leading to a homoclinic orbit. The paper
\cite{Simo} contains numerical evidence of such homoclinic orbits for the
first number of the larger of these parameters $\mu_{k}^{\ast}$, and gives an
analytic proof for sufficiently small $\mu_{k}^{\ast}$.

The aim of this section is to show that using our method it is possible to
obtain rigorous enclosures of the stable and unstable manifolds, and to
validate the existence of homoclinic orbits for the large values of $\mu
_{k}^{\ast}$. We focus on the largest of the parameters%
\[
\mu_{2}^{\ast}\approx0.004253863522
\]
and prove that
\[
\mu_{2}^{\ast}\in0.004253863522+10^{-10}\left[  -1,1\right]  .
\]
The established homoclinic connection is depicted in Figure
\ref{fig:homoclinic}.

\begin{remark}
Our estimate on the parameter for which we have a homoclinic orbit to
$L_{1}^{\mu_{2}^{\ast}}$ is very tight. This is thanks to the fact that our
method for establishing invariant manifolds produces very tight rigorous
bounds. This demonstrates that it is a tool that can successfully be applied
for nontrivial problems.
\end{remark}

\begin{remark}
Our paper focuses on $\mu_{2}^{\ast}$ since it is the largest parameter, hence
furthest away from the analytic proof of \cite{Simo}. Using our method one can
obtain a proof also for other parameters. As the parameters become smaller
though, the proof becomes more challenging numerically.
\end{remark}

\begin{figure}[ptb]
\begin{center}
\includegraphics[height=4cm]{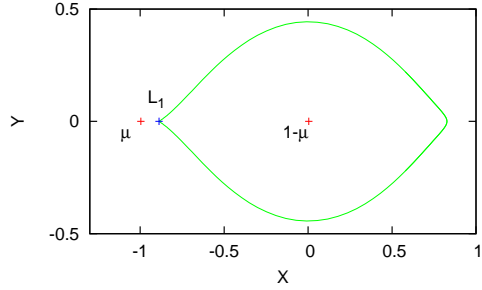}
\end{center}
\caption{Homoclinic orbit in green, the masses in red, and the fixed point
$L_{1}$ in blue.}%
\label{fig:homoclinic}%
\end{figure}


\section{Cones and horizontal discs\label{sec:cones-discs}}

In this section we give some auxiliary results, which are then used in the
proofs of Theorems \ref{th:main-unstable}, \ref{th:main-stable} in section
\ref{sec:construction}.

We start with some simple facts which follow straight from
(\ref{eq:cones-def1}--\ref{eq:cones-def2}). We formulate this as a remark, and
give the proof in the appendix.

\begin{remark}
\label{rem:cones-setup}

\begin{enumerate}
\item \label{rem:cones-setup1}If $\left\Vert x\right\Vert \leq1$ and
$Q_{h}\left(  x,y\right)  \geq\alpha_{h}-1,$ then $\left\Vert y\right\Vert
\leq1.$

\item \label{rem:cones-setup2}If $\left\Vert y\right\Vert \leq1$ and
$Q_{v}\left(  x,y\right)  \leq1-\alpha_{v}$ then $\left\Vert x\right\Vert
\leq1$.

\item \label{rem:cones-setup2a}If $Q_{h}\left(  x,y\right)  \geq\alpha_{h}-1$
and $Q_{v}\left(  x,y\right)  \leq1-\alpha_{v}$ then $(x,y)\in N$.

\item \label{rem:cones-setup3}If $\left\Vert y\right\Vert \leq a$ then
$Q_{h}\left(  x,y\right)  \geq-a^{2}$.
\end{enumerate}
\end{remark}

\begin{proof}
The proof is given in Appendix \ref{app:cones-setup}.
\end{proof}

We now give two technical lemmas.

\begin{lemma}
\label{lem:positive-cone-back-conv}Assume that $(q_{0},q_{-1},q_{-2},\ldots)$
is a backward trajectory in $\{Q_{h}\geq0\}\cap N$. If $f$ satisfies cone conditions for
$(Q_{v},m_{v}),$ then for $C=\sqrt{2\left(  1-\alpha_{v}\alpha_{h}\right)
^{-1}}$ and any $k\leq0$
\[
\left\Vert q_{k}\right\Vert \leq C\left(  \sqrt{m_{v}}\right)  ^{k}.
\]

\end{lemma}

\begin{proof}
The proof is given in Appendix \ref{app:positive-cone-back-conv}.\textbf{ }
\end{proof}

\begin{lemma}
\label{lem:negative-cone-forward-conv}Assume that for a $q_{0}\in N$, for all
$k\geq0,$ $f^{k}(q_{0})\in\{Q_{v}\leq0\}\cap N.$ If $f$ satisfies cone
conditions for $\left(  Q_{h},m_{h}\right)  ,$ then for $C=\sqrt{2\left(
1-\alpha_{v}\alpha_{h}\right)  ^{-1}}$ and any $k\geq0$
\[
\left\Vert f^{k}(q_{0})\right\Vert \leq C\left(  \sqrt{m_{h}}\right)  ^{k}.
\]

\end{lemma}

\begin{proof}
The proof is given in Appendix \ref{app:negative-cone-forward-conv}.
\end{proof}

We now introduce a notion of a horizontal disc. Horizontal discs will be the
building blocks in our construction of the invariant manifolds.

\begin{definition}
\label{def:hd} Let $Q(x,y)=\alpha\left\Vert x\right\Vert ^{2}-\beta\left\Vert
y\right\Vert ^{2}$ for $\alpha,\beta>0$. Let $h:\overline{B}_{u}%
\rightarrow\mathbb{R}^{u+s}$ be a continuous mapping. We say that $h$ is a
$Q$-horizontal disc if
\begin{align}
Q(h(x_{1})-h(x_{2}))  &  >0\text{ for any }x_{1}\neq x_{2},\label{eq:cc1}\\
\pi_{x}h(0)  &  =0. \label{eq:cc3}%
\end{align}

\end{definition}

\begin{figure}[ptb]
\begin{center}
\includegraphics[width=6cm]{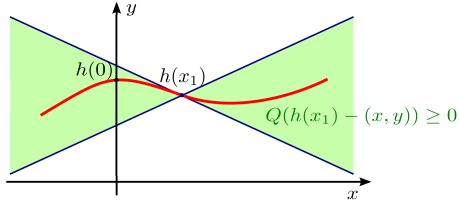}
\end{center}
\caption{A $Q$-horizontal disc $h$ in red. For any point $x_{1}\in
\overline{B}_{u}$ the disc $h$ lies within the interior of a cone attached at
$h(x_{1})$.}%
\label{fig:hQ}%
\end{figure}

\begin{definition}
\label{def:hd-in-N}We say that a $Q$-horizontal disc is in $N$ if
$h(B_{u})\subset N.$
\end{definition}

\begin{definition}
\label{def:hd-radius}Let $c>0$. We say that a $Q$-horizontal disc has radius
$c$ if%
\begin{equation}
Q(h\left(  \partial B_{u}\right)  )=c, \label{eq:cc2}%
\end{equation}

\end{definition}

\begin{figure}[ptb]
\begin{center}
\includegraphics[width=6cm]{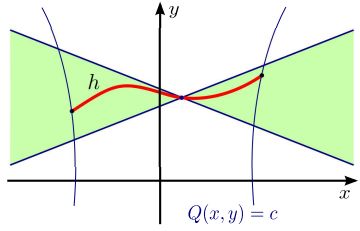}
\end{center}
\caption{A $Q$-horizontal disc $h$ with radius $c$ (in red). The image
$\partial B_{u}$ is contained in the set $\left\{  Q=c\right\}  $.}%
\label{fig:hQc}%
\end{figure}

Let $Q(x,y)=\alpha\left\Vert x\right\Vert ^{2}-\beta\left\Vert y\right\Vert
^{2}$ for $\alpha,\beta>0$. The following lemmas are consequences of
Definition \ref{def:hd}.

\begin{lemma}
\label{lem:h-injective}If $h$ is a $Q$-horizontal disc, then $\pi_{x}\circ h$
is bijective onto its image.
\end{lemma}

\begin{proof}
Take any $x_{1},x_{2}\in\overline{B}_{u}$ and suppose that $\pi_{x}%
(h(x_{1}))=\pi_{x}(h(x_{2}))$. Then
\[
Q(h(x_{1})-h(x_{2}))=-\beta ||\pi_{y}(h(x_{1}))-\pi_{y}(h(x_{2}))||^{2}\leq0
\]
The condition (\ref{eq:cc1}) implies that $x_{1}=x_{2}$. It means that
$\pi_{x}\circ h$ is injective, and as a consequence it is bijective onto its image.
\end{proof}

\begin{lemma}
\label{lem:hor-disc-proj}If $h$ is a $Q$-horizontal disc of radius $c$, then
for any $x^{\ast}\in\overline{B}_{u}\left(  0,\sqrt{\frac{c}{\alpha}}\right)
,$ there exists a unique $x$ such that $\pi_{x}h(x)=x^{\ast}$.
\end{lemma}

\begin{proof}
By definition, $h$ is continuous. By Lemma \ref{lem:h-injective}, $\pi
_{x}h:\overline{B}_{u}\rightarrow\mathbb{R}^{u}$ is injective. This means that
$\pi_{x}h\left(  B_{u}\right)  $ is homeomorphic to a ball in $\mathbb{R}^{u}$.

For any $x\in\partial B_{u}$ 
\[
c=Q(h(x))=\alpha\left\Vert \pi_{x}h(x)\right\Vert ^{2}-\beta\left\Vert \pi
_{y}h(x)\right\Vert ^{2}\leq\alpha\left\Vert \pi_{x}h(x)\right\Vert ^{2},
\]
hence $\left\Vert \pi_xh(x)\right\Vert \geq\sqrt{\frac{c}{\alpha}}.$ This
means that $\partial\left[  \pi_xh\left(  B_u\right)  \right]  \cap B_u\left(
0,\sqrt{\frac{c}{\alpha}}\right)  =\emptyset,$ hence either 
\[
\pi_{x}h\left(  B_{u}\right)  \cap B_{u}\left(  0,\sqrt{\frac{c}{\alpha}%
}\right)  =\emptyset,
\]
or%
\begin{equation}
B_{u}\left(  0,\sqrt{\frac{c}{\alpha}}\right)  \subset\pi_{x}h\left(
B_{u}\right)  .\label{eq:Bh-in-ball}%
\end{equation}
Since $\pi_{x}h(0)=0\in B_{u}\left(  0,\sqrt{\frac{c}{\alpha}}\right)  ,$ we
see that (\ref{eq:Bh-in-ball}) must be the case. From (\ref{eq:Bh-in-ball}),
by continuity of $h$,
\[
\overline{B}_{u}\left(  0,\sqrt{\frac{c}{\alpha}}\right)  \subset\pi
_{x}h\left(  \overline{B}_{u}\right)  .
\]
We have thus shown that for any $x^{\ast}$ there exists an $x$ such that
$\pi_{x}h(x)=x^{\ast}$. Such point needs to be unique since for $x_{1}\neq
x_{2}$
\begin{align*}
0 &  <\frac{1}{\alpha}Q\left(  h(x_{1})-h(x_{2})\right)  \\
&  =\left\Vert \pi_{x}h(x_{1})-\pi_{x}h(x_{2})\right\Vert ^{2}-\frac{\beta
}{\alpha}\left\Vert \pi_{y}h(x_{1})-\pi_{y}h(x_{2})\right\Vert ^{2}\\
&  \leq\left\Vert \pi_{x}h(x_{1})-\pi_{x}h(x_{2})\right\Vert ^{2}.
\end{align*}

\end{proof}

Let $Q(x,y)=\alpha\left\Vert x\right\Vert ^{2}-\beta\left\Vert y\right\Vert
^{2}$ for $\alpha,\beta>0$, and let $c^{\ast}>0$. In the
following arguments we shall use the function\textbf{ }$\phi:\mathbb{R}%
^{u}\times\mathbb{R}^{s}\rightarrow\mathbb{R}^{u}\times\mathbb{R}^{s}$\textbf{
}%
\begin{equation}
\phi(u,s)=\left\{
\begin{array}
[c]{ll}%
\left(  u\sqrt{\frac{1}{\alpha}\left(  c^{\ast}+\beta\left\Vert s\right\Vert
^{2}\right)  },s\right)   & \qquad\text{if }\left\Vert u\right\Vert \leq1,\\
\left(  u\left[  \frac{1}{\left\Vert u\right\Vert }\left(  \sqrt{\frac
{1}{\alpha}\left(  c^{\ast}+\beta\left\Vert s\right\Vert ^{2}\right)
}-1\right)  +1\right]  ,s\right)   & \qquad\text{if }\left\Vert u\right\Vert
>1,
\end{array}
\right.  \label{eq:phi-def}%
\end{equation}
which will be used as a suitable change of coordinates. (Note that $\phi$ is
continuous.) The choice of $\phi$ is motivated by the fact that $\{\phi
(u,s):\left\Vert u\right\Vert \leq1\}=\left\{  Q\leq c^{\ast}\right\}  $.
Thus, we can say that $\phi$ \textquotedblleft straightens out" $\left\{
Q\leq c^{\ast}\right\}  $ (see Figure \ref{fig:fhphi}). We now give a
technical lemma.\begin{figure}[ptb]
\begin{center}
\includegraphics[width=11cm]{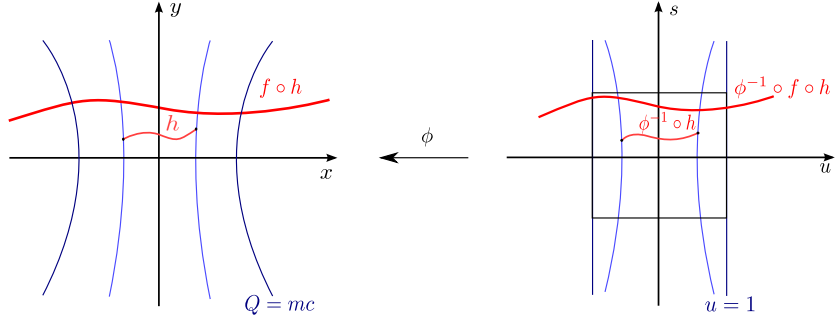}
\end{center}
\caption{The change of coordinates $\phi$ applied to a $Q$-horizontal disc $h$
and to $f\circ h$.}%
\label{fig:fhphi}%
\end{figure}

\begin{lemma}
\label{lem:phi-prop}If $s_{1},s_{2}\in\mathbb{R}^{s}$ and $u\in\mathbb{R}^{u}$
then%
\[
Q\left(  \phi(u,s_{1})-\phi(u,s_{2})\right)  \leq0.
\]

\end{lemma}

\begin{proof}
For $a,b>0$
\[
x\rightarrow\sqrt{b+a\left\Vert x\right\Vert ^{2}}%
\]
is Lipschitz with constant $\sqrt{a},$ thus
\begin{align}
\left\Vert \pi_{x}\left(  \phi(u,s_{1})-\phi(u,s_{2})\right)  \right\Vert  &
\leq\sqrt{\frac{\beta}{\alpha}}\left\Vert s_{1}-s_{2}\right\Vert
,\label{eq:lip-phix}\\
\left\Vert \pi_{y}\left(  \phi(u,s_{1})-\phi(u,s_{2})\right)  \right\Vert  &
=\left\Vert s_{1}-s_{2}\right\Vert .\nonumber
\end{align}
This gives that for any $s_{1},s_{2}$%
\begin{equation}
Q\left(  \phi(u,s_{1})-\phi(u,s_{2})\right)  \leq\alpha\left(  \sqrt
{\frac{\beta}{\alpha}}\left\Vert s_{1}-s_{2}\right\Vert \right)  ^{2}%
-\beta\left\Vert s_{1}-s_{2}\right\Vert ^{2}=0,\label{eq:phi-cone-prop}%
\end{equation}
as required.
\end{proof}

The following lemma is a key result that will be used in our construction of
the manifolds.

\begin{lemma}
\label{lem:hor-graph-transform}Let $Q(x,y)=\alpha\left\Vert x\right\Vert
^{2}-\beta\left\Vert y\right\Vert ^{2}$ for $\alpha,\beta>0$. Let
$h:\overline{B}_{u}\rightarrow N$ be a $Q$-horizontal disc in $N$ of radius
$c>0.$ Let $m>0.$ If $f$ satisfies cone conditions for $(Q,m),$ then for any
$c^{\ast}\in(0,mc]$ there exists a $Q$-horizontal disc $h^{\ast}%
:\overline{B}_{u}\rightarrow\mathbb{R}^{u+s}$ of radius $c^{\ast}$, such that
\begin{equation}
h^{\ast}(\overline{B}_{u})=f\circ h(\overline{B}_{u})\cap\{Q\leq c^{\ast}\},
\label{eq:fh-cond}%
\end{equation}
and%
\begin{equation}
\pi_{u}\phi^{-1}(h^{\ast}(u))=u. \label{eq:h-star-on-u}%
\end{equation}

\end{lemma}

\begin{figure}[ptb]
\begin{center}
\includegraphics[width=4.5cm]{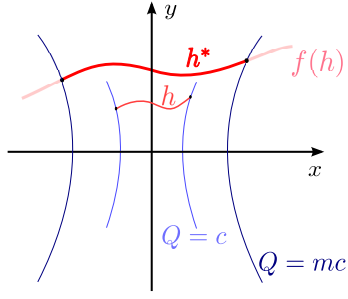}
\end{center}
\caption{$Q$-horizontal disc $h^{*}$ with radius $c^{*}=mc$ obtained as an
intersection of an image under $f$ of a $Q$-horizontal disc $h$ of radius $c$,
and the set $\{Q\leq c^{*}\}$ (see Lemma \ref{lem:hor-graph-transform}).}%
\label{fig:h_ast}%
\end{figure}

\begin{proof}
Let $h_{\lambda}(x):=(\pi_{x}h(x),\lambda\pi_{y}h(x))$ and let us define the
function%
\[
H:[0,1]\times\overline{B}_{u}\rightarrow\lbrack0,1]\times\mathbb{R}^{u},
\]%
\[
H(\lambda,x)=\left(  \lambda,\pi_{u}\phi^{-1}\left(  f(h_{\lambda}(x))\right)
\right)  .
\]

We will show that $H$ is an open map. Observe that $h_{\lambda}$ are
$Q$-horizontal discs in $N$. Let $x_{1},x_{2}\in\overline{B}_{u}$ and
$x_{1}\neq x_{2}$. By the fact that $f$ satisfies cone conditions, $Q\left(
f(h_{\lambda}(x_{1}))-f(h_{\lambda}(x_{2}))\right)  >0$, and by Lemma
\ref{lem:phi-prop} we can not have $\pi_{u}\phi^{-1}(f(h_{\lambda}%
(x_{1})))=\pi_{u}\phi^{-1}(f(h_{\lambda}(x_{1}))).$ Hence $H$ is injective. By
definition, $H$ is also continuous, thus it is an open map.

We consider the set $[0,1]\times \mathbb{R}^{u}$, with topology induced from
$\mathbb{R}\times\mathbb{R}^{u}$. Let $A=[0,1]\times B_{u}$. Note that $A$ is
open in $[0,1]\times\mathbb{R}^{u}$, that $clA=[0,1]\times\overline{B}_{u}$
and $\partial A=\left[  0,1\right]  \times\partial B_{u}$. Since $H$ is an
open map, $H(A)\cap clA$ is open in $clA.$ We will show that $H(A)\cap clA$ is
also closed in $clA$.

Take any $x\in\partial B_{u}$. Since $Q(h_{\lambda}(x))\geq Q(h(x))=c$, 
by the fact that $f$ satisfies cone conditions for $(Q,m)$
\[
Q(f(h_{\lambda}(x)))>mQ(h_{\lambda}(x))\geq mc\geq c^{\ast}.
\]
Hence $\pi_{u}\phi^{-1}\left(  f(h_{\lambda}(\partial B_{u}))\right)
\cap\overline{B}_{u}=\emptyset,$ which means that for all $\lambda\in
\lbrack0,1]$%
\begin{equation}
H\left(  \partial A\right)  \cap\mathrm{cl}A=H(\left[  0,1\right]
\times\partial B_{u})\cap\mathrm{cl}A=\emptyset. \label{eq:homotopy-1}%
\end{equation}
We thus see that $H(A)\cap\mathrm{cl}A$ is closed in $\mathrm{cl}A$.

Since $H(A)\cap\mathrm{cl}A$ is both open and closed in $\mathrm{cl}A$, we
either have%
\[
H(A)\cap\mathrm{cl}A=\mathrm{cl}A,
\]
or%
\begin{equation}
H(A)\cap\mathrm{cl}A=\emptyset. \label{eq:HA-empty}%
\end{equation}
Since $h_{\lambda=0}(0)=0,$ $f\left(  0\right)  =0\ $and $\phi^{-1}%
(0)=0$,\textbf{ }%
\[
H\left(  0,0\right)  =\left(  0,\pi_{u}\phi^{-1}\left(  f(h_{\lambda
=0}(0))\right)  \right)  =\left(  0,0\right)  \in A.
\]
We see that we can not have (\ref{eq:HA-empty}), hence $\mathrm{cl}A\subset
H(A)$. This in particular implies that $\left\{  1\right\}  \times\overline{B}_{u}\subset
H(\{1\}\times B_{u})$, hence%
\begin{equation}
\overline{B}_{u}\subset\pi_{u}\phi^{-1}\left(  f(h(\overline{B}_{u}))\right)
. \label{eq:inclusion-phi}%
\end{equation}

From (\ref{eq:inclusion-phi}) we see that for any $u\in\overline{B}_{u}$ there
exists an $x=x(u)\in B_{u},$ such that%
\begin{equation}
u=\pi_{u}\phi^{-1}\left(  f(h(x(u)))\right)  . \label{eq:xu-implicit}%
\end{equation}
We now define%
\begin{equation}
h^{\ast}(u)=f(h(x(u))). \label{eq:h-star-def}%
\end{equation}
Note that from (\ref{eq:xu-implicit}) and (\ref{eq:h-star-def}) follows
(\ref{eq:h-star-on-u}).

For $h^{\ast}$ to be well defined we need to show that the choice of $x(u)$ is
unique. Assume that for $x_{1}\neq x_{2}$ we have%
\begin{align*}
\phi^{-1}\left(  f(h(x_{1}))\right)   &  =\left(  u,s_{1}\right) \\
\phi^{-1}\left(  f(h(x_{2}))\right)   &  =\left(  u,s_{2}\right)
\end{align*}
with $s_{1}\neq s_{2}$. From Lemma \ref{lem:phi-prop} we know that%
\[
Q\left(  \phi(u,s_{1})-\phi(u,s_{2})\right)  \leq0.
\]
On the other hand,
\begin{equation}
Q\left(  \phi(u,s_{1})-\phi(u,s_{2})\right)  =Q\left(  f(h(x_{1}%
))-f(h(x_{2}))\right)  >mQ\left(  h(x_{1})-h(x_{2})\right)  >0.
\label{eq:temp-unique}%
\end{equation}
We obtain a contradiction, hence we must have $s_{1}=s_{2}.$ This shows that
$h^{\ast}$ is well defined.

We need to show that $h^{\ast}$ is a $Q$-horizontal disc of radius $c^{\ast}$.
We first show (\ref{eq:cc1}). Observe that (\ref{eq:xu-implicit}) implies that
$x(u_{1})\neq x(u_{2})$ for any $u_{1}\neq u_{2}$. From (\ref{eq:h-star-def})
and by the fact that $f$ satisfies cone conditions for $(Q,m)$
\begin{align*}
Q(h^{\ast}(x_{1})-h^{\ast}(x_{2}))  &  =Q\left(  f(h(x(u_{1})))-f(h(x(u_{2}%
)))\right) \\
&  >mQ\left(  h(x(u_{1}))-h(x(u_{2}))\right) \\
&  >0.
\end{align*}
Now we prove (\ref{eq:cc3}). From (\ref{eq:h-star-on-u}), $\phi^{-1}\left(
h^{\ast}\left(  0\right)  \right)  =\left(  0,s\right)  ,$ for some
$s\in\mathbb{R}^{s}.$ This gives
\[
\pi_{x}h^{\ast}\left(  0\right)  =0\sqrt{\frac{1}{\alpha}\left(  c^{\ast
}+\beta\left\Vert s\right\Vert ^{2}\right)  }=0.
\]
Now we prove that $Q(h^{\ast}(\partial B_{u}))=c^{\ast}.$ Assume that
$u\in\partial B_{u}.$ By (\ref{eq:h-star-on-u}) we know that%
\[
\phi^{-1}(h^{\ast}(u))=\left(  u,s\right)
\]
for some $s\in\mathbb{R}^{s}$. Since $\left\Vert u\right\Vert =1$%
\begin{align*}
Q\left(  h^{\ast}(u)\right)   &  =Q\left(  \phi\left(  u,s\right)  \right) \\
&  =Q\left(  u\sqrt{\frac{1}{\alpha}\left(  c^{\ast}+\beta\left\Vert
s\right\Vert ^{2}\right)  },s\right) \\
&  =\alpha\left\Vert u\sqrt{\frac{1}{\alpha}\left(  c^{\ast}+\beta\left\Vert
s\right\Vert ^{2}\right)  }\right\Vert ^{2}-\beta\left\Vert s\right\Vert
^{2}\\
&  =c^{\ast}.
\end{align*}

The fact that (\ref{eq:fh-cond}) holds, follows from our construction of
$h^{\ast}$.
\end{proof}

\begin{lemma}
\label{lem:unstable-tool}Assume that $h$ is a $Q_{h}$-horizontal disc in $N$.
Assume also that $h$ is a $Q_{v}$-horizontal disc of radius $c=1-\alpha_{v}$
and that $h(0)=0.$ Assume that $f$ satisfies cone conditions for $(Q_{h}%
,m_{h})$ and $(Q_{v},m_{v}),$ where $m_{v}>1$ and $m_{h}>0$. Let $h^{\ast}$ be the $Q_{v}$-horizontal disc
of radius $c^{\ast}=c$ from Lemma \ref{lem:hor-graph-transform}. Then
$h^{\ast}(0)=0,$ and $h^{\ast}$ is a $Q_{h}$-horizontal disc in $N.$
\end{lemma}

\begin{proof}
Since $h^{\ast}$ is a $Q_{v}$ horizontal disc, $\pi_{x}h^{\ast}(0)=0.$ For any
$x\neq0$%
\[
\left\Vert \pi_{x}f\circ h(x)\right\Vert ^{2}\geq Q_{v}(f\circ h(x))=Q_{v}%
(f\circ h(x)-f\circ h(0))>m_{v}Q_{v}(h(x)-h(0))>0.
\]%
Since by (\ref{eq:fh-cond}) $h^{\ast}(0)=f\circ h(x_{0})$ for some $x_{0}%
\in\overline{B}_{u},$ and since
$\pi_xh^\ast(0)=0,$ we see that $x_0=0.$ This gives%
\[
h^{\ast}(0)=f\circ h(0)=f(0)=0.
\]

Since for any $x_{1}\neq x_{2}$, $x_{1},x_{2}\in\overline{B}_{u}%
$
\[
Q_{h}(f(h(x_{1}))-f(h(x_{2})))>m_{h}Q_{h}(h(x_{1})-h(x_{2}))>0,
\]
hence by (\ref{eq:fh-cond}), $h^{\ast}$ is a $Q_{h}$-horizontal disc.

We need to show that $h^{\ast}$ is contained in $N.$ Observe that since
$h^{\ast}$ is a $Q_{h}$-horizontal disc and since $\alpha_{h}\in\left(
0,1\right)  $
\begin{equation}
Q_{h}(h^{\ast}(x))=Q_{h}(h^{\ast}(x)-h^{\ast}(0))\geq0>\alpha_{h}-1.
\label{eq:temp-h-in-N-1}%
\end{equation}
Since $h^{\ast}$ is a $Q_{v}$-horizontal disc of radius $c^{\ast}%
=c=1-\alpha_{v}$
\begin{equation}
Q_{v}(h^{\ast}(x))\leq1-\alpha_{v}. \label{eq:temp-h-in-N-2}%
\end{equation}
The fact that $h^{\ast}(x)$ is contained in N follows from
(\ref{eq:temp-h-in-N-1}), (\ref{eq:temp-h-in-N-2}) and point
\ref{rem:cones-setup2a} from Remark \ref{rem:cones-setup}.
\end{proof}

\begin{figure}[ptb]
\begin{center}
\includegraphics[width=4.5cm]{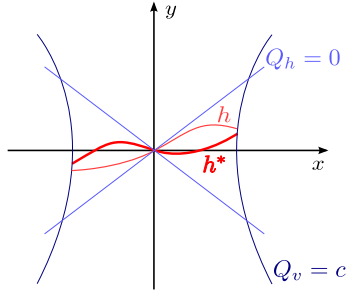}
\end{center}
\caption{$Q_{h}$-horizontal disc $h^{*}$ obtained from $h$ in terms of Lemma
\ref{lem:unstable-tool}.}%
\label{fig:mm1}%
\end{figure}

\begin{lemma}
\label{lem:stable-tool}Assume that $0<m_{h},m_{v}$ and $m_{h}<1$. Assume that
$h$ is a $Q_{v}$-horizontal disc of radius $c\leq1-\alpha_{v}$ in $N$ and that
$Q_{h}(h(\overline{B}_{u}))\geq\alpha_{h}-1.$ Assume also that $f$ satisfies
cone conditions for $(Q_{h},m_{h})$ and $(Q_{v},m_{v})$. Let $h^{\ast}$ be the
$Q_{v}$-horizontal disc of radius $c^{\ast}=\min\{m_{v}c,c\}$ from Lemma
\ref{lem:hor-graph-transform}. Then $Q_{h}(h^{\ast}(\overline{B}_{u}%
))\geq\alpha_{h}-1,$ and $h^{\ast}$ is a $Q_{v}$-horizontal disc in $N.$
\end{lemma}

\begin{proof}
Since $\alpha_{h}\in\left(  0,1\right)  ,$ $m_{h}\in(0,1)$ and $f$ satisfies
cone conditions for $(Q_{h},m_{h}),$ for any $x\in\overline{B}_{u}$%
\[
Q_{h}(f\circ h(x))\geq m_{h}Q_{h}(h(x))\geq m_{h}\left(  \alpha_{h}-1\right)
\geq\alpha_{h}-1,
\]
which by (\ref{eq:fh-cond}) proves that $Q_{h}(h^{\ast}(\overline{B}_{u}%
))\geq\alpha_{h}-1$.

The fact that $h^{\ast}$ is in $N$ follows from the facts that $Q_{h}(h^{\ast
}(\overline{B}_{u}))\geq\alpha_{h}-1$ and $Q_{v}\left(  h^{\ast}%
(\overline{B}_{u})\right)  \leq c^{\ast}=\min\{m_{v}c,c\}\leq c\leq
1-\alpha_{v}$, combined with point \ref{rem:cones-setup2a} from Remark
\ref{rem:cones-setup}. \begin{figure}[ptb]
\begin{center}
\includegraphics[width=5.5cm]{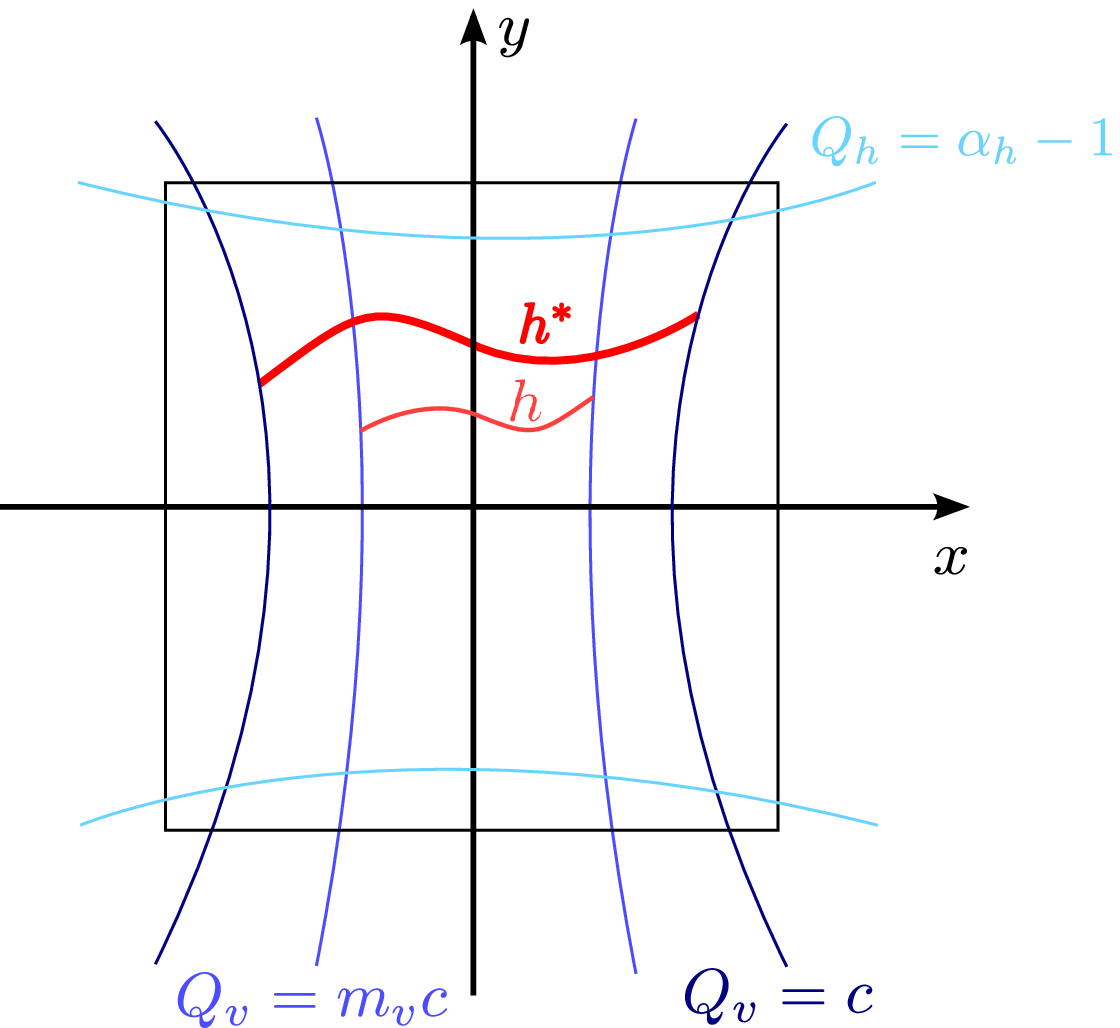}
\end{center}
\caption{$Q_{h}$-horizontal disc $h^{\ast}$ obtained from $h$ in terms of
Lemma \ref{lem:stable-tool}.}%
\label{fig:0mm}%
\end{figure}
\end{proof}


\section{Construction of the stable and unstable
manifolds\label{sec:construction}}

In this section we give proofs of Theorems \ref{th:main-unstable},
\ref{th:main-stable}.

\subsection{Proof of Theorem \ref{th:main-unstable}}

\begin{proof}
We start by considering two points $q^{\ast},q^{\ast\ast}\in W_{\sqrt{m_{v}%
},N}^{u}$, with backward trajectories $(q_{0}^{\ast},q_{-1}^{\ast},\ldots)$
and $(q_{0}^{\ast\ast},q_{-1}^{\ast\ast},\ldots).$ We will show that:%
\begin{equation}
\pi_{x}q^{\ast}=\pi_{x}q^{\ast\ast}\implies q^{\ast}=q^{\ast\ast}.
\label{eq:q-stars-eq}%
\end{equation}
Since $q^{\ast},q^{\ast\ast}\in W_{\sqrt{m_{v}},N}^{u}$, for any $k\leq0$%
\begin{equation}
\left\Vert q_{k}^{\ast}\right\Vert \leq C\left(  \sqrt{m_{v}}\right)
^{k},\qquad\left\Vert q_{k}^{\ast\ast}\right\Vert \leq C\left(  \sqrt{m_{v}%
}\right)  ^{k}. \label{eq:temp-contraction-str}%
\end{equation}
On the other hand, since $Q_{h}\left(  q^{\ast}-q^{\ast\ast}\right)
=-\left\Vert \pi_{y}\left(  q^{\ast}-q^{\ast\ast}\right)  \right\Vert ^{2}%
\leq0,$ by $\left(  Q_{h},m_{h}\right)  $ cone conditions we see that for any
$k<0\ $\textbf{ }%
\begin{align}
0  &  \geq Q_{h}\left(  q^{\ast}-q^{\ast\ast}\right) \nonumber\\
&  =Q_{h}\left(  q_{0}^{\ast}-q_{0}^{\ast\ast}\right) \nonumber\\
&  =Q_{h}\left(  f(q_{-1}^{\ast})-f(q_{-1}^{\ast\ast})\right)
\label{eq:temp-contr-zero}\\
&  \geq m_{h}Q_{h}\left(  q_{-1}^{\ast}-q_{-1}^{\ast\ast}\right) \nonumber\\
&  \geq\ldots\nonumber\\
&  \geq m_{h}^{\left\vert k\right\vert }Q_{h}\left(  q_{k}^{\ast}-q_{k}%
^{\ast\ast}\right)  .\nonumber
\end{align}
This implies that for $k\leq0$%
\begin{equation}
\left\Vert q_{k}^{\ast}-q_{k}^{\ast\ast}\right\Vert ^{2}\geq\left\vert
Q_{h}\left(  q_{k}^{\ast}-q_{k}^{\ast\ast}\right)  \right\vert \geq m_{h}%
^{k}\left\vert Q_{h}\left(  q^{\ast}-q^{\ast\ast}\right)  \right\vert \geq0.
\label{eq:temp-exp-str}%
\end{equation}
Since $m_{v}>m_{h},$ (\ref{eq:temp-contraction-str}) and
(\ref{eq:temp-exp-str}) implies that $q^{\ast}=q^{\ast\ast},$ which proves
(\ref{eq:q-stars-eq}).

We now move to the construction of the function $w^{u}$ from Theorem
\ref{th:main-unstable}. Let us define a mapping $h_{0}:\overline{B_{u}%
}\rightarrow N$, as $h_{0}(x)=(x\sqrt{1-\alpha_{v}},0)$. Then $h_{0}$ is a
$Q_{h}$-horizontal disc in $N$ and $Q_{v}$-horizontal disk of radius
$c=1-\alpha_{v}$. Moreover, for any $x\in\overline{B_{u}}$
\[
Q_{h}\left(  h_{0}(x)\right)  =\alpha_{h}||x\sqrt{1-\alpha_{v}}||^{2}%
\geq0>\alpha_{h}-1,
\]
which means that assumptions of Lemma \ref{lem:unstable-tool} are satisfied.
Applying inductively Lemma \ref{lem:unstable-tool}, we obtain a sequence of
$Q_{h}$-horizontal discs in $N$, that are also $Q_{v}$-horizontal discs of
radius $c,$ which we shall denote as $h_{i}$, for $i=0,1,\ldots$. These
horizontal disks are given by $h_{i+1}=h_{i}^{\ast}$ in terms of Lemma
\ref{lem:hor-graph-transform}.

We will show that for any $x^{\ast}\in\overline{B}_{u}(0,r^{u})$ there exists
a unique point $q^{\ast}$ such that $\pi_{x}q^{\ast}=x^{\ast}$, which lies in
$W_{\sqrt{m_{v}},N}^{u}.$ By Lemma \ref{lem:hor-disc-proj}, for any $i\geq0$
there exists a point $p_{i}^{\ast}\in h_{i}(\overline{B}_{u})$ such that
$\pi_{x}p_{i}^{\ast}=x^{\ast}$. Since $N$ is compact, there exists a
convergent subsequence $p_{i_{l}}^{\ast}$ to a point $q^{\ast}(x^{\ast})$%
\begin{equation}
\lim_{l\rightarrow\infty}p_{i_{l}}^{\ast}=q^{\ast}\left(  x^{\ast}\right)  ,
\label{eq:il-seq}%
\end{equation}
with $\pi_{x}q^{\ast}(x^{\ast})=x^{\ast}$. We will show that such point is
unique, and that it lies in $W_{\sqrt{m_{v}},N}^{u}$. Such point will be the
candidate for $w^{u}(x^{\ast})=\pi_{y}q^{\ast}(x^{\ast})$.

We start by showing that there exists a backward trajectory $(q_{0}^{\ast
},q_{-1}^{\ast},\ldots)$ in $N\cap\{Q_{h}\geq0\}$ reaching $q^{\ast}(x^{\ast
})$. It is sufficient to show that for any $n\geq0$ there exists a
$q_{-n}^{\ast}$ such that $f^{k}(q_{-n}^{\ast})\in N\cap\{Q_{h}\geq0\}$ for
$k=0,\ldots,n$ and $f^{n}(q_{-n}^{\ast})=q^{\ast}(x^{\ast})$. Let $i\geq n$.
Since
\[
p_{i}^{\ast}\in h_{i}\left(  \overline{B}_{u}\right)  =f(h_{i-1}(\overline
{B}_{u}))\cap\left\{  Q_{v}\leq c\right\}  ,
\]
we see that $p_{i}^{\ast}=f(p_{i,-1}^{\ast}),$ with $p_{i,-1}^{\ast}\in
h_{i-1}(\overline{B}_{u})$. Since $h_{i-1}$ is a $Q_{h}$-horizontal disc and
since $h_{i-1}(0)=0$, we have $p_{i,-1}^{\ast}\in h_{i-1}(\overline{B}%
_{u})\subset\{Q_{h}\geq0\}$. Similarly, since%
\[
p_{i-1}^{\ast}\in h_{i-1}\left(  \overline{B}_{u}\right)  =f(h_{i-2}%
(\overline{B}_{u}))\cap\left\{  Q_{v}\leq c\right\}  ,
\]
we obtain a point $p_{i,-2}^{\ast}\in h_{i-2}(\overline{B}_{u})\subset
\{Q_{h}\geq0\}$ such that $f(p_{i,-2}^{\ast})=p_{i,-1}^{\ast}$. Proceeding
inductively we obtain a point $p_{i,-n}^{\ast}$ such that $f^{k}%
(p_{i,-n}^{\ast})\in N\cap\{Q_{h}\geq0\}$ and $f^{n}(p_{i,-n}^{\ast}%
)=p_{i}^{\ast}$. Consider now the subsequence $p_{i_{l},-n}^{\ast}$, in terms
of $l$, where $i_{l}$ is the subsequence form (\ref{eq:il-seq}). Since
$N\cap\{Q_{h}\geq0\}$ is compact, there exists a convergent subsequence
$p_{i_{l_{m}},-n}^{\ast}$ to a point $q_{-n}^{\ast}$%
\[
\lim_{m\rightarrow\infty}p_{i_{l_{m}},-n}^{\ast}=q_{-n}^{\ast}.
\]
observing that%
\begin{align*}
f^{k}(q_{-n}^{\ast}) &  =\lim_{m\rightarrow\infty}f^{k}(p_{i_{l_{m}},-n}%
^{\ast})\in N,\\
f^{n}(q_{-n}^{\ast}) &  =\lim_{m\rightarrow\infty}f^{n}(p_{i_{l_{m}},-n}%
^{\ast})=\lim_{m\rightarrow\infty}p_{i_{l_{m}}}^{\ast}=q^{\ast}(x^{\ast}),
\end{align*}
we achieve our goal of proving existence of $q_{-n}^{\ast}$. Thus, there
exists a backward trajectory in $N\cap\{Q_{h}\geq0\}$ reaching $q^{\ast
}(x^{\ast}).$

Since $q^{\ast}(x^{\ast})$ has a backward trajectory in $N\cap\{Q_{h}\geq0\}$,
by Lemma \ref{lem:positive-cone-back-conv} we see that $q^{\ast}(x^{\ast})\in
W_{\sqrt{m_{v}},N}^{u}$.

We now show that the point $q^{\ast}(x^{\ast})$ from (\ref{eq:il-seq}) is
unique. If we take another point $q^{\ast\ast}(x^{\ast})$, then both points
$q^{\ast}(x^{\ast})$ and $q^{\ast\ast}(x^{\ast})$ are in $W_{\sqrt{m_{v}}%
,N}^{u}$ and thus by (\ref{eq:q-stars-eq}) they must coincide. This means
that
\[
w^{u}(x^{\ast})=\pi_{y}q^{\ast}(x^{\ast})
\]
is well defined.

From our construction, for any $x_{1}^{\ast},x_{2}^{\ast}\in\overline{B}%
_{u}(0,r^{u})$ we have
\[
q^{\ast}(x_{1}^{\ast})=\lim_{k\rightarrow\infty}p_{1,k}^{\ast},\qquad
\text{and\qquad}q^{\ast}(x_{2}^{\ast})=\lim_{k\rightarrow\infty}p_{2,k}^{\ast
},
\]
for sequences $p_{1,k}^{\ast},p_{2,k}^{\ast}\in h_{k}(\overline{B}_{u})$.
Thus
\[
Q_{h}\left(  q^{\ast}(x_{1}^{\ast})-q^{\ast}(x_{2}^{\ast})\right)
=\lim_{k\rightarrow\infty}Q_{h}\left(  p_{1,k}^{\ast}-p_{2,k}^{\ast}\right)
\geq0.
\]
This implies that%
\[
\alpha_{h}||x_{1}^{\ast}-x_{2}^{\ast}||^{2}-||w^{u}(x_{1}^{\ast})-w^{u}%
(x_{2}^{\ast})||^{2}=Q_{h}(q^{\ast}(x_{1}^{\ast})-q^{\ast}(x_{2}^{\ast}%
))\geq0,
\]
which proves that $w^{u}$ is Lipschitz with constant $L=\sqrt{\alpha_{h}}$.
\end{proof}

\begin{remark}
\label{rem:C-bound-from-cones}In the proof the constant $C$ from Theorem
\ref{th:main-unstable} was established via Lemma
\ref{lem:positive-cone-back-conv}. In Lemma \ref{lem:positive-cone-back-conv}
we see that $C=\sqrt{2\left(  1-\alpha_{v}\alpha_{h}\right)  ^{-1}}$ depends
only on the coefficients of the cones $\alpha_{h},\alpha_{v}$.
\end{remark}

\subsection{Proof of Theorem \ref{th:main-stable}}

\begin{proof}
Let us fix $y_{0}$ such that $||y_{0}||\leq\sqrt{1-\alpha_{h}}.$ We define a
mapping $h_{0}:\overline{B_{u}}\rightarrow N$, as $h_{0}(x)=(x\sqrt
{1-\alpha_{v}(1-\left\Vert y_{0}\right\Vert ^{2})},y_{0})$. Then $h_{0}$ is a
$Q_{v}$-horizontal disk in $N$ of radius $c=1-\alpha_{v}>0$. By point
\ref{rem:cones-setup3} from Remark \ref{rem:cones-setup}, for any
$x\in\overline{B_{u}}$\textbf{ }%
\[
Q_{h}\left(  x,y_{0}\right)  \geq-\left(  \sqrt{1-\alpha_{h}}\right)
^{2}=\alpha_{h}-1,
\]
which means that assumptions of Lemma \ref{lem:stable-tool} are satisfied.
Applying inductively Lemma \ref{lem:stable-tool}, we obtain a sequence of
$Q_{v}$-horizontal discs, which we shall denote as $h_{i}$, for $i=0,1,\ldots
$. These horizontal disks are given by $h_{i+1}=h_{i}^{\ast}$ in terms of
Lemma \ref{lem:hor-graph-transform}. The $h_{i}$ are also $Q_{h}$-horizontal
discs for $i>0$.

By construction, we know that $\pi_{x}h_{i}(0)=0$. Let $x_{i}^{\ast}$ be a
point such that $f^{i}(h_{0}(x_{i}^{\ast}))=h_{i}(0)$. Since $f$ satisfies
cone conditions for $(Q_{v},m_{v})$, for any point $q$ such that
$Q_{v}(f(q))\leq0$, we must have $Q_{v}(q)\leq0$. This means that since
$Q_{v}(h_{i}(0))\leq0$, we also have $Q_{v}(f^{k}(h_{0}(x_{i}^{\ast})))\leq0$
for $k=0,\ldots,i$. Since $\overline{B}_{u}$ is compact, there exists a
convergent subsequence $x_{i_{m}}^{\ast}$ to some $x^{\ast}\in\overline{B}%
_{u}$. It means that there exists an $x^{\ast}\in\overline{B}_{u}$ such
that\textbf{ }%
\begin{equation}
f^{i}(h_{0}(x^{\ast}))\in\left\{  Q_{v}\leq0\right\}  \cap N\quad\text{for all
}i\geq0. \label{eq:temp-x*-in-N}%
\end{equation}
The point $x^{\ast}$ is a candidate for $w^{s}(y_{0}).$

We now check that $w^{s}(y_{0})$ is well defined. Suppose that\textbf{ }%
\begin{equation}
f^{i}(h_{0}(x^{\ast\ast}))\in\left\{  Q_{v}\leq0\right\}  \cap N\quad\text{for
all }i\geq0. \label{eq:temp-x**-in-N}%
\end{equation}
From Lemma \ref{lem:negative-cone-forward-conv} we know that for
$C>0,$\textbf{ }%
\begin{equation}
\left\Vert f^{i}(h_{0}(x^{\ast}))\right\Vert \leq C\left(  \sqrt{m_{h}%
}\right)  ^{i},\qquad\left\Vert f^{i}(h_{0}(x^{\ast\ast}))\right\Vert \leq
C\left(  \sqrt{m_{h}}\right)  ^{i}. \label{eq:w0-stable}%
\end{equation}
On the other hand,%
\begin{align*}
\left\Vert f^{i}(h_{0}(x^{\ast}))-f^{i}(h_{0}(x^{\ast\ast}))\right\Vert ^{2}
&  \geq Q_{v}(f^{i}(h_{0}(x^{\ast}))-f^{i}(h_{0}(x^{\ast\ast})))\\
&  \geq m_{v}^{i}Q_{h}(h_{0}(x^{\ast})-h_{0}(x^{\ast\ast}))\\
&  \geq0.
\end{align*}
Since $m_{h}<m_{v}$ above inequality and (\ref{eq:w0-stable}) imply that
$h_{0}(x^{\ast})=h_{0}(x^{\ast\ast})$.

The same argument can be used to show that any two points $p^{\ast}\neq
p^{\ast\ast}$ on the strong stable manifold $W_{\sqrt{m_{h}}}^{s}$ must
satisfy
\begin{equation}
Q_{v}\left(  p^{\ast}-p^{\ast\ast}\right)  \leq0, \label{eq:in-neg-cone}%
\end{equation}
since if this were not the case, we would have%
\[
\left\Vert f^{i}(p^{\ast})-f^{i}(p^{\ast\ast})\right\Vert ^{2}\geq Q_{v}%
(f^{i}(p^{\ast})-f^{i}(p^{\ast\ast}))>m_{v}^{i-1}Q_{h}(f(p^{\ast}%
)-f(x^{\ast\ast}))>0,
\]
contradicting contraction at the rate $\sqrt{m_{h}}$.

Observe that by (\ref{eq:w0-stable}), $w^{s}$ parameterizes the stable manifold.

It left to show that $w^{s}$ is Lipschitz with a constant $L=\sqrt{\alpha_{v}%
}$. By (\ref{eq:in-neg-cone})
\[
0\geq Q_{v}((w^{s}(y_{1}),y_{1})-(w^{s}(y_{2}),y_{2})),
\]
hence
\[
||w^{s}(y_{1})-w^{s}(y_{2})||^{2}\leq\alpha_{v}||y_{1}-y_{2}||,
\]
as required.
\end{proof}

\begin{remark}
In the proof, the constant $C$ from Theorem \ref{th:main-stable} was
established via Lemma \ref{lem:negative-cone-forward-conv}. In Lemma
\ref{lem:negative-cone-forward-conv} we see that $C=\sqrt{2\left(
1-\alpha_{v}\alpha_{h}\right)  ^{-1}}$ depends only on the coefficients of the
cones $\alpha_{h},\alpha_{v}$.
\end{remark}


\section{Establishing manifolds of ODEs\label{sec:flows}}

In this section we consider an ODE%
\begin{equation}
p^{\prime}=F(p), \label{eq:ode}%
\end{equation}
with $F$ of class $C^{1}$, satisfying: for all $p\in N$
\begin{equation}
\left\Vert F\left(  p\right)  \right\Vert \leq\mu,\qquad\left\Vert
DF(p)\right\Vert \leq L, \label{eq:vect-fld-as-1}%
\end{equation}
and for any $p_{1},p_{2}\in N$%
\begin{equation}
\left\Vert DF\left(  p_{1}\right)  -DF\left(  p_{2}\right)  \right\Vert \leq
M\left\Vert p_{1}-p_{2}\right\Vert . \label{eq:vect-fld-as-2}%
\end{equation}
Let $\phi_{t}(p)$ stand for the flow induced by (\ref{eq:ode}). We assume that
zero is a fixed point.

\begin{definition}
Let $U$ be a neighborhood of zero and let $\lambda>0$. We say that a set
$W_{\lambda,U}^{u}$ consisting of all points $p$ satisfying:

\begin{enumerate}
\item $\phi_{t}(p)\in U$ for all $t\leq0;$

\item there exists a constant $C>0$ (which can depend on $p$), such that for
all $t\leq0$,%
\[
\left\Vert \phi_{t}(p)\right\Vert \leq Ce^{t\lambda};
\]

\end{enumerate}

\noindent is a strong unstable manifold with expansion rate $\lambda$ in $U$.
\end{definition}

\begin{definition}
Let $U$ be a neighborhood of zero and let $\lambda<0$. We say that a set
$W_{\lambda,U}^{s}$ consisting of all points $p$ satisfying:

\begin{enumerate}
\item $\phi_{t}(p)\in U$ for all $t\geq0;$

\item there exists a constant $C>0$ (which can depend on $p$), such that for
all $t\geq0$,%
\begin{equation}
\left\Vert \phi_{t}(p)\right\Vert \leq Ce^{t\lambda};
\label{eq:contr-rate-cond}%
\end{equation}

\end{enumerate}

\noindent is a strong stable manifold with contraction rate $\lambda$ in $U.$
\end{definition}

Let us assume that
\[
\left[  DF(N)\right]  \subset\left(
\begin{array}
[c]{cc}%
\mathbf{A} & \boldsymbol{\varepsilon}_{1}\\
\boldsymbol{\varepsilon}_{2} & \mathbf{B}%
\end{array}
\right)  ,
\]
where $\mathbf{A}$, $\mathbf{B}$, $\boldsymbol{\varepsilon}_{1}$ and
$\boldsymbol{\varepsilon}_{2}$ are interval matrices. Let $Q_{h}$ and $Q_{v}$
be as defined in (\ref{eq:cones-def1}--\ref{eq:cones-def2}). Assume that we
have two constants $c_{h},c_{v}\in\mathbb{R}$ such that for any $A\in
\mathbf{A},$ $B\in\mathbf{B}$, $\varepsilon_{1}\in\boldsymbol{\varepsilon}%
_{1}$ and $\varepsilon_{2}\in\boldsymbol{\varepsilon}_{2}$%
\begin{align}
x^{T}\left(  A-\frac{1}{2}\left(  \left\Vert \varepsilon_{1}\right\Vert
+\frac{1}{\alpha_{h}}\left\Vert \varepsilon_{2}\right\Vert \right)
\mathrm{Id}\right)  x  &  >c_{h}\left\Vert x\right\Vert ^{2}%
,\label{eq:ode-ass1}\\
x^{T}\left(  A-\frac{1}{2}\left(  \left\Vert \varepsilon_{1}\right\Vert
+\alpha_{v}\left\Vert \varepsilon_{2}\right\Vert \right)  \mathrm{Id}\right)
x  &  >c_{v}\left\Vert x\right\Vert ^{2},\label{eq:ode-ass2}\\
y^{T}\left(  B+\frac{1}{2}\left(  \left\Vert \varepsilon_{2}\right\Vert
+\alpha_{h}\left\Vert \varepsilon_{1}\right\Vert \right)  \mathrm{Id}\right)
y  &  <c_{h}\left\Vert y\right\Vert ^{2},\label{eq:ode-ass3}\\
y^{T}\left(  B+\frac{1}{2}\left(  \left\Vert \varepsilon_{2}\right\Vert
+\frac{1}{\alpha_{v}}\left\Vert \varepsilon_{1}\right\Vert \right)
\mathrm{Id}\right)  y  &  <c_{v}\left\Vert y\right\Vert ^{2}.
\label{eq:ode-ass4}%
\end{align}

\begin{theorem}
\label{th:odes-wu} Let $r^{u}=\sqrt{1-\alpha_{v}}$ and $U=\overline{B}%
_{u}(0,r^{u})\times\overline{B}_{s}.$ If $c_{v}>c_{h}$ and $c_{v}>0,$ then
there exists a function $w^{u}:\overline{B}_{u}(0,r^{u})\rightarrow
\overline{B}_{s},$ such that%
\[
W_{c_{v},N}^{u}\cap U=\left\{  (x,w^{u}(x))|x\in\overline{B}_{u}%
(0,r^{u})\right\}  .
\]
Moreover, $w^{u}$ is Lipschitz with a constant $L_{u}=\sqrt{\alpha_{h}}$.
\end{theorem}

\begin{theorem}
\label{th:odes-ws} Let $r^{s}=\sqrt{1-\alpha_{h}}$ and $U=\overline{B}%
_{u}\times\overline{B}_{s}(0,r^{s}).$ If $c_{h}<c_{v}$ and $c_{h}<0$, then
there exists a function $w^{s}:\overline{B}_{s}(0,r^{s})\rightarrow
\overline{B}_{u},$ such that%
\[
W_{c_{h},N}^{s}\cap U=\left\{  (w^{s}(y),y)|y\in B_{s}(0,r^{s})\right\}  .
\]
Moreover, $w^{s}$ is Lipschitz with a constant $L_{s}=\sqrt{\alpha_{v}}.$
\end{theorem}

\begin{remark}
Let us note that the assumptions of Theorems \ref{th:odes-wu} and
\ref{th:odes-ws} follow directly from the estimates on the vector field. There
is no need to integrate the ODE to verify them.
\end{remark}

We need some auxiliary results before we give proofs of the theorems at the
end of the section. We start with a technical lemma.

\begin{lemma}
\label{lem:gronwell-est}Assume that the vector field satisfies the conditions
(\ref{eq:vect-fld-as-1}) and (\ref{eq:vect-fld-as-2}). Then for
\begin{align*}
g_{1}(p_{1},p_{2},t) &  =\phi_{t}(p_{1})-\phi_{t}(p_{2})-\left(  p_{1}%
-p_{2}\right)  ,\\
g_{2}(p_{1},p_{2},t) &  =F(\phi_{t}(p_{1}))-F(\phi_{t}(p_{2}))-\left(
F(p_{1})-F(p_{2})\right)  ,
\end{align*}
we have the following estimates%
\begin{align}
\left\Vert g_{1}\left(  p_{1},p_{2},t\right)  \right\Vert  &  \leq\left(
e^{\left\vert t\right\vert L}-1\right)  \left\Vert p_{1}-p_{2}\right\Vert
,\label{eq:gronw-est-1}\\
\left\Vert g_{2}\left(  p_{1},p_{2},t\right)  \right\Vert  &  \leq\left(
L\left(  e^{L\left\vert t\right\vert }-1\right)  +\left\vert t\right\vert
e^{L\left\vert t\right\vert }\mu M\right)  \left\Vert p_{1}-p_{2}\right\Vert
.\label{eq:gronw-est-2}%
\end{align}

\end{lemma}

\begin{proof}
The proof is given in Appendix \ref{app:gronwell-est}.
\end{proof}

The following lemma will be the key for the proof of Theorems \ref{th:odes-wu}
and \ref{th:odes-ws}.

\begin{lemma}
\label{lem:cc-flows}Let $Q(x,y)=\alpha\left\Vert x\right\Vert ^{2}%
-\beta\left\Vert y\right\Vert ^{2}$ with $\alpha,\beta>0$. Assume that for
$c\in\mathbb{R}$ and any $A\in\mathbf{A},$ $B\in\mathbf{B}$, $\varepsilon
_{1}\in\boldsymbol{\varepsilon}_{1}$ and $\varepsilon_{2}\in
\boldsymbol{\varepsilon}_{2}$ holds%
\begin{align}
x^{T}\left(  A-\frac{1}{2}\left(  \left\Vert \varepsilon_{1}\right\Vert
+\frac{\beta}{\alpha}\left\Vert \varepsilon_{2}\right\Vert \right)
\mathrm{Id}\right)  x  &  >c\left\Vert x\right\Vert ^{2}, \label{eq:flows-cc1}%
\\
y^{T}\left(  B+\frac{1}{2}\left(  \left\Vert \varepsilon_{2}\right\Vert
+\frac{\alpha}{\beta}\left\Vert \varepsilon_{1}\right\Vert \right)
\mathrm{Id}\right)  y  &  <c\left\Vert y\right\Vert ^{2}. \label{eq:flows-cc2}%
\end{align}
Then for sufficiently small $t>0,$ the map $\phi_{t}$ satisfies cone
conditions in $N$ for $(Q,m=1+t2c).$
\end{lemma}

\begin{remark}
Conditions (\ref{eq:flows-cc1}), (\ref{eq:flows-cc2}) hold when the two
matrixes%
\begin{align*}
&  A-\frac{1}{2}\left(  \left\Vert \varepsilon_{1}\right\Vert +\frac{\beta
}{\alpha}\left\Vert \varepsilon_{2}\right\Vert +2c\right)  \mathrm{Id},\\
-  &  B+\frac{1}{2}\left(  -\left\Vert \varepsilon_{2}\right\Vert
-\frac{\alpha}{\beta}\left\Vert \varepsilon_{1}\right\Vert +2c\right)
\mathrm{Id}\mathbf{,}%
\end{align*}
are strictly positive definite. The same approach can be used to verify
(\ref{eq:ode-ass1})--(\ref{eq:ode-ass4}).
\end{remark}

\begin{proof}
(of Lemma \ref{lem:cc-flows}) Let $g_{1}$ and $g_{2}$ be the functions defined
in Lemma \ref{lem:gronwell-est}. Let $\mathcal{Q}$ denote the $\left(
u+s\right)  \times\left(  u+s\right)  $ matrix associated with $Q$, that is,
$Q(p)=p^{T}\mathcal{Q}p$, and let
\[
C=\int_{0}^{1}DF\left(  \left(  1-t\right)  p_{2}+tp_{1}\right)  dt\in\left[
DF(N)\right]  .
\]
We can compute%
\begin{align}
&  \frac{d}{dt}Q\left(  \phi_{t}(p_{1})-\phi_{t}(p_{2})\right)
\label{eq:temp-flows1}\\
&  =\left(  \phi_{t}^{\prime}(p_{1})-\phi_{t}^{\prime}(p_{2})\right)
^{T}\mathcal{Q}\left(  \phi_{t}(p_{1})-\phi_{t}(p_{2})\right) \nonumber\\
&  \quad+\left(  \phi_{t}(p_{1})-\phi_{t}(p_{2})\right)  ^{T}\mathcal{Q}%
\left(  \phi_{t}^{\prime}(p_{1})-\phi_{t}^{\prime}(p_{2})\right) \nonumber\\
&  =\left(  F(\phi_{t}(p_{1}))-F(\phi_{t}(p_{2}))\right)  ^{T}\mathcal{Q}%
\left(  \phi_{t}(p_{1})-\phi_{t}(p_{2})\right) \nonumber\\
&  \quad+\left(  \phi_{t}(p_{1})-\phi_{t}(p_{2})\right)  ^{T}\mathcal{Q}%
\left(  F(\phi_{t}(p_{1}))-F(\phi_{t}(p_{2}))\right) \nonumber\\
&  =\left(  F(p_{1})-F(p_{2})+g_{2}\left(  p_{1},p_{2},t\right)  \right)
^{T}\mathcal{Q}\left(  p_{1}-p_{2}+g_{1}\left(  p_{1},p_{2},t\right)  \right)
\nonumber\\
&  \quad+\left(  p_{1}-p_{2}+g_{1}\left(  p_{1},p_{2},t\right)  \right)
^{T}\mathcal{Q}\left(  F(p_{1})-F(p_{2})+g_{2}\left(  p_{1},p_{2},t\right)
\right) \nonumber\\
&  =\left(  p_{1}-p_{2}\right)  ^{T}\left(  C^{T}\mathcal{Q}+\mathcal{Q}%
C\right)  \left(  p_{1}-p_{2}\right)  +g_{3}(p_{1},p_{2},t),\nonumber
\end{align}
where by (\ref{eq:gronw-est-1}--\ref{eq:gronw-est-2}) we see that for any
$p_{1},p_{2}\in N$ and $\left\vert t\right\vert \leq1$
\[
\left\Vert g_{3}(p_{1},p_{2},t)\right\Vert \leq bt\left\Vert p_{1}%
-p_{2}\right\Vert ^{2},
\]
for a constant $b$ dependent on $\mu,L,M,\alpha$ and $\beta$.

Since $C\in\left[  DF(N)\right]  $, it is of the form%
\[
C=\left(
\begin{array}
[c]{cc}%
A & \varepsilon_{1}\\
\varepsilon_{2} & B
\end{array}
\right)  ,
\]
with $A\in\mathbf{A},$ $B\in\mathbf{B},$ $\varepsilon_{1}\in
\boldsymbol{\varepsilon}_{1}$ and $\varepsilon_{2}\in\boldsymbol{\varepsilon
}_{2}$. Using the fact that%
\[
x^{T}\varepsilon_{i}y\geq-\left\Vert \varepsilon_{i}\right\Vert \left\Vert
x\right\Vert \left\Vert y\right\Vert \geq-\frac{1}{2}\left\Vert \varepsilon
_{i}\right\Vert \left(  x^{T}x+y^{T}y\right)  \quad\text{for }i=1,2
\]
for $p=\left(  x,y\right)  \neq0$ we can compute%
\begin{align}
p^{T}\mathcal{Q}Cp  &  =\alpha x^{T}Ax+\alpha x^{T}\varepsilon_{1}y-\beta
y^{T}\varepsilon_{2}x-\beta y^{T}By\label{eq:temp-flows2}\\
&  \geq\alpha x^{T}Ax-\alpha\frac{1}{2}\left\Vert \varepsilon_{1}\right\Vert
\left(  x^{T}x+y^{T}y\right) \nonumber\\
&  -\beta\frac{1}{2}\left\Vert \varepsilon_{2}\right\Vert \left(  x^{T}%
x+y^{T}y\right)  -\beta y^{T}By\nonumber\\
&  =\alpha x^{T}\left(  A-\frac{1}{2}\left(  \left\Vert \varepsilon
_{1}\right\Vert +\frac{\beta}{\alpha}\left\Vert \varepsilon_{2}\right\Vert
\right)  \mathrm{Id}\right)  x\nonumber\\
&  -\beta y^{T}\left(  B+\frac{1}{2}\left(  \left\Vert \varepsilon
_{2}\right\Vert +\frac{\alpha}{\beta}\left\Vert \varepsilon_{1}\right\Vert
\right)  \mathrm{Id}\right)  y\nonumber\\
&  >\alpha c\left\Vert x\right\Vert ^{2}-\beta c\left\Vert y\right\Vert
^{2}\nonumber\\
&  =cp^{T}\mathcal{Q}p.\nonumber
\end{align}
Similarly, it follows that for $p\neq0$%
\begin{equation}
p^{T}C^{T}\mathcal{Q}p>cp^{T}\mathcal{Q}p. \label{eq:temp-flows3}%
\end{equation}

Combining (\ref{eq:temp-flows1}), (\ref{eq:temp-flows2}) and
(\ref{eq:temp-flows3}), taking $p_{1}\neq p_{2}$, for some $\xi\in\left[
-t,t\right]  $ (which depends on $p_{1},p_{2}$ and $t$),%
\begin{align*}
&  Q\left(  \phi_{t}(p_{1})-\phi_{t}(p_{2})\right) \\
&  =Q\left(  \phi_{t}(p_{1})-\phi_{t}(p_{2})\right)  |_{t=0}+t\left.  \frac
{d}{ds}Q\left(  \phi_{s}(p_{1})-\phi_{s}(p_{2})\right)  \right\vert _{s=\xi}\\
&  =Q\left(  p_{1}-p_{2}\right)  +t\left(  p_{1}-p_{2}\right)  ^{T}\left(
C^{T}\mathcal{Q}+\mathcal{Q}C\right)  \left(  p_{1}-p_{2}\right)
+tg_{3}(p_{1},p_{2},\xi)\\
&  >\left(  1+2tc\right)  Q\left(  p_{1}-p_{2}\right)  +tg_{3}(p_{1},p_{2}%
,\xi).
\end{align*}
Since
\[
\left\Vert tg_{3}(p_{1},p_{2},\xi)\right\Vert \leq bt^{2}\left\Vert
p_{1}-p_{2}\right\Vert ^{2},
\]
we see that for sufficiently small $\left\vert t\right\vert $%
\[
Q\left(  \phi_{t}(p_{1})-\phi_{t}(p_{2})\right)  \geq\left(  1+2tc\right)
Q\left(  p_{1}-p_{2}\right)  ,
\]
as required.
\end{proof}

We are now ready to prove Theorems \ref{th:odes-wu} and \ref{th:odes-ws}.

\begin{proof}
[Proof of Theorem \ref{th:odes-wu}]From Lemma \ref{lem:cc-flows} it follows
that there exists a $\tau^{\ast}$ such that for any $\tau\in\left(
0,\tau^{\ast}\right)  $ the time shift along the trajectory map $\phi_{\tau}$
satisfies cone conditions for $\left(  Q_{h},m_{h}\right)  $ and $\left(
Q_{v},m_{v}\right)  ,$ with%
\begin{align*}
m_{h} &  =m_{h}(\tau)=1+\tau2c_{h},\\
m_{v} &  =m_{v}(\tau)=1+\tau2c_{v}.
\end{align*}
We can choose $\tau^{\ast}$ small enough so that $m_{h}(\tau)>0,$ for $\tau
\in\left(  0,\tau^{\ast}\right)  $. Also, since $c_{v}>0$, we see that
$m_{v}(\tau)>1$. By Theorem \ref{th:main-unstable}, there exists a strong
unstable manifold for $\phi_{\tau}$, such that $W_{\sqrt{m_{v}(\tau)},N}%
^{u}\cap U$ is a graph of a function $w^{u,\tau}:\overline{B}_{u}%
(0,r)\rightarrow\overline{B}_{s}.$

We will now show that for $\tau_{1},\tau_{2}\in\left(  0,\tau^{\ast}\right)  $
we have $w^{u,\tau_{1}}=w^{u,\tau_{2}}$. Assume that $\tau_{1}<\tau_{2}$. Let
us fix $x\in\overline{B}_{u}(0,r)$ and define $p_{1}=w^{u,\tau_{1}}(x),$
$p_{2}=w^{u,\tau_{2}}(x)$. We will show that $p_{1}=p_{2}$. In our argument we
will use the fact that
\begin{equation}
1\leq\left(  1+\frac{b}{a}\right)  ^{\frac{a}{2}}\leq\left(  1+b\right)
^{\frac{1}{2}}\qquad\text{for }a\in(0,1],\text{ and }b>0.
\label{eq:tmp-rate-ab}%
\end{equation}
Let $n\in\mathbb{N}$ be fixed. Since $\tau_{1}<\tau_{2}$, there exists a
$k\in\mathbb{N}$, $k>n$ and $\delta\in\left[  0,\tau_{1}\right)  $, such that%
\[
n\tau_{2}=k\tau_{1}-\delta.
\]
From (\ref{eq:tmp-rate-ab}), by taking $b=2c_{v}\tau_{1}$ and $a=\frac{n}{k},$
follows that%
\[
\left(  1+2c_{v}\frac{k}{n}\tau_{1}\right)  ^{-\frac{n}{2}}\geq\left(
1+2c_{v}\tau_{1}\right)  ^{-\frac{k}{2}},
\]
which gives%
\begin{multline*}
\left(  \sqrt{m_{v}(\tau_{2})}\right)  ^{-n}=\left(  1+2c_{v}\tau_{2}\right)
^{-\frac{n}{2}}\geq\left(  1+2c_{v}\frac{k}{n}\tau_{1}\right)  ^{-\frac{n}{2}%
}\\
\geq\left(  1+2c_{v}\tau_{1}\right)  ^{-\frac{k}{2}}=\left(  \sqrt{m_{v}%
(\tau_{1})}\right)  ^{-k}.
\end{multline*}
From this estimate we see that%
\begin{multline*}
\left\Vert \left(  \phi_{\tau_{2}}\right)  ^{-n}\left(  p_{1}\right)
\right\Vert =\left\Vert \phi_{-n\tau_{2}}\left(  p_{1}\right)  \right\Vert
=\left\Vert \phi_{\delta}\circ\phi_{-k\tau_{1}}\left(  p_{1}\right)
\right\Vert \\
\leq e^{L\delta}\left\Vert \phi_{-k\tau_{1}}\left(  p_{1}\right)  \right\Vert
\leq e^{L\tau_{1}}C\sqrt{m_{v}(\tau_{1})}^{-k}\leq e^{L\tau_{1}}C\sqrt
{m_{v}(\tau_{2})}^{-n},
\end{multline*}
which means that $p_{1}$ is on the strong unstable manifold for the map
$\phi_{\tau_{2}}$, hence $p_{1}=p_{2},$ as required.

Since the strong unstable manifold for the time shift maps $\phi_{\tau}$ is
independent of the choice of $\tau$, we see that it coincides with a strong
unstable manifold for the flow $\phi_{t}$. What remains is to prove that the
expansion rate for this manifold is $c_{v}$.

For $\tau\in\left(  0,\tau^{\ast}\right)  $ the map $\phi_{\tau}$ satisfies
cone conditions for $\left(  Q_{h},m_{h}(\tau)\right)  $ and $\left(
Q_{v},m_{v}(\tau)\right)  $ (where $Q_{h}$ and $Q_{v}$ are the same for all
$\tau$), hence by Remark \ref{rem:C-bound-from-cones},%
\[
\left\Vert \left(  \phi_{\tau}\right)  ^{-n}(p)\right\Vert \leq C\sqrt
{m_{v}(\tau)}^{-n},
\]
for $C$ which is independent of $\tau$. Let $t<0$. The expansion rate
condition follows by computing\textbf{ }%
\begin{align}
\left\Vert \phi_{t}(p)\right\Vert  &  =\left\Vert \left(  \phi_{\frac
{\left\vert t\right\vert }{n}}\right)  ^{-n}(p)\right\Vert \nonumber\\
&  \leq C\left(  \sqrt{m_{v}\left(  \left\vert t\right\vert /n\right)
}\right)  ^{-n}\label{eq:flow-expansion}\\
&  =C\left(  1+\frac{2}{n}\left\vert t\right\vert c_{v}\right)  ^{-\frac{n}%
{2}}\overset{n\rightarrow+\infty}{\rightarrow}Ce^{tc_{v}},\nonumber
\end{align}
as required.
\end{proof}

\begin{proof}
[Proof of Theorem \ref{th:odes-ws}]The result follows from combining Lemma
\ref{lem:cc-flows} with Theorem \ref{th:main-stable}, and mirror arguments to
the proof of Theorem \ref{th:odes-wu}.
\end{proof}


\section{Proof of a homoclinic connection in the restricted three body
problem\label{sec:3bp-application}}

\subsection{A suitable change of coordinates}

To verify assumptions of Theorem \ref{th:odes-wu} close to $L_{1}$ we consider
the PCR3BP in suitable local coordinates. These are introduced below in two
steps. The first step takes the linearized vector field into a Jordan form,
through a linear change of coordinates. The second step involves a nonlinear
change of coordinates, which further \textquotedblleft straightens out" the
unstable coordinate.

We now discuss the linear change of coordinates. The libration point is of the
form%
\[
L_{1}^{\mu}=\left(  x_{L_{1}}^{\mu},0,0,x_{L_{1}}^{\mu}\right)  .
\]
The Jacobian of the vector field has an unstable eigenvalue, which we denote
as $\lambda.$ We consider the following linear change of coordinates
(\cite{Jorba}, Section~2.1)
\begin{equation}
C^{\mu}=\left(
\begin{array}
[c]{cccc}%
\frac{2\lambda}{s_{1}} & \frac{-2\lambda}{s_{1}} & 0 & \frac{2v}{s_{2}}\\
\frac{\lambda^{2}-2c_{2}-1}{s_{1}} & \frac{\lambda^{2}-2c_{2}-1}{s_{1}} &
\frac{-v^{2}-2c_{2}-1}{s_{2}} & 0\\
\frac{\lambda^{2}+2c_{2}+1}{s_{1}} & \frac{\lambda^{2}+2c_{2}+1}{s_{1}} &
\frac{-v^{2}+2c_{2}+1}{s_{2}} & 0\\
\frac{\lambda^{3}+(1-2c_{2})\lambda}{s_{1}} & \frac{-\lambda^{3}%
-(1-2c_{2})\lambda}{s_{1}} & 0 & \frac{-v^{3}+(1-2c_{2})v}{s_{2}}%
\end{array}
\right)  \label{eq:C-mu}%
\end{equation}
where%
\begin{align*}
c_{2}  &  =\frac{1}{\gamma^{3}}\left(  \mu+\frac{(1-\mu)\gamma^{3}}%
{(1-\gamma)^{3}}\right)  ,\\
\gamma &  =x_{L_{1}}^{\mu}+1-\mu,\\
s_{1}  &  =\sqrt{2\lambda\left(  \left(  4+3c_{2}\right)  \lambda^{2}%
+4+5c_{2}-6c_{2}^{2}\right)  },\\
s_{2}  &  =\sqrt{v\left(  \left(  4+3c_{2}\right)  v^{2}-4-5c_{2}+6c_{2}%
^{2}\right)  },
\end{align*}
that puts the linear terms of the vector field at $L_{1}^{\mu}$ into the
Jordan form%
\[
\left(
\begin{array}
[c]{cccc}%
\lambda & 0 & 0 & 0\\
0 & -\lambda & 0 & 0\\
0 & 0 & 0 & v\\
0 & 0 & -v & 0
\end{array}
\right)  .
\]
We note that in the above, for sake of keeping the notations short, we have
omitted the dependence of parameters on $\mu.$ In fact, for different $\mu,$
each nonzero entry of $C^{\mu}$ is different.

Using the notation $\mathbf{x}=\left(  X,Y,P_{X},P_{Y}\right)  $ for the
original coordinates of the problem, we introduce local coordinates
$\mathbf{v}$ at $L_{1}^{\mu}$ as%

\[
\mathbf{x}=L_{1}^{\mu}+C^{\mu}\mathbf{v.}%
\]
In coordinates $\mathbf{v}$, the vector field is%
\[
\tilde{F}(\mathbf{v})=\left(  C^{\mu}\right)  ^{-1}F\left(  L_{1}^{\mu}%
+C^{\mu}\mathbf{v}\right)  ,
\]
and the Jacobian of the vector field at zero is%
\[
D\tilde{F}(0)=\text{diag}\left(  A_{h},A_{c}\right)  ,
\]
with
\[
A_{h}=%
\begin{pmatrix}
\lambda & 0\\
0 & -\lambda
\end{pmatrix}
\qquad\text{and }\qquad A_{c}=\left(
\begin{array}
[c]{cc}%
0 & v\\
-v & 0
\end{array}
\right)  .
\]
The matrix $A_{h}$ represents the linearized hyperbolic dynamics, and $A_{c}$
represents the center rotation at the fixed point.

The second step is to consider a nonlinear change of coordinates. To do so let
us consider an equation
\begin{equation}
\tilde{F}(K(\mathsf{x}))=R(\mathsf{x})\,DK(\mathsf{x}), \label{eq:cohomology}%
\end{equation}
where $K:\mathbb{R}\rightarrow\mathbb{R}^{4}$ and $R:\mathbb{R}\rightarrow
\mathbb{R}$ are analytic. We refer to (\ref{eq:cohomology}) as the
\emph{cohomology equation}. The graph of $K$ parametrizes the unstable
manifold at the fixed point. An approximate solution of $K$ and $R$ can be
found numerically (for details see \cite{Llave}). We use a polynomial $K$,
which is an approximate, numerically obtained solution of (\ref{eq:cohomology}%
), and use it to define the following nonlinear change of coordinates%
\[
\psi=\left(  \psi_{0},\psi_{1},\psi_{2},\psi_{3}\right)  :\mathbb{R}%
^{4}\rightarrow\mathbb{R}^{4},
\]
where%
\begin{align}
\psi_{0}\left(  \mathsf{x},\mathsf{y}_{1},\mathsf{y}_{2},\mathsf{y}%
_{3}\right)   &  =K_{0}(\mathsf{x})-\left(  \mathsf{y}_{1}K_{1}^{\prime
}(\mathsf{x})+\mathsf{y}_{2}K_{2}^{\prime}(\mathsf{x})+\mathsf{y}_{3}%
K_{3}^{\prime}(\mathsf{x})\right)  ,\label{eq:psi-form}\\
\psi_{i}\left(  \mathsf{x},\mathsf{y}_{1},\mathsf{y}_{2},\mathsf{y}%
_{3}\right)   &  =K_{i}(\mathsf{x})+\mathsf{y}_{i}K_{0}^{\prime}%
(\mathsf{x})\quad\text{for }i=1,2,3.\nonumber
\end{align}
Note that since the graph of $K$ approximates the unstable manifold,
$\psi(\mathsf{x},0)=K(\mathsf{x})$ gives points close to the unstable manifold
of the fixed point. The intuitive idea behind (\ref{eq:psi-form}) is to
arrange the coordinates so that $\psi\left(  \mathsf{x},\mathsf{y}%
_{1},\mathsf{y}_{2},\mathsf{y}_{3}\right)  -K(\mathsf{x})$ is orthogonal to
$K^{\prime}(\mathsf{x})$ (see Figure \ref{fig:psi-coord}).

\begin{figure}[ptb]
\begin{center}
\includegraphics[height=2.5cm]{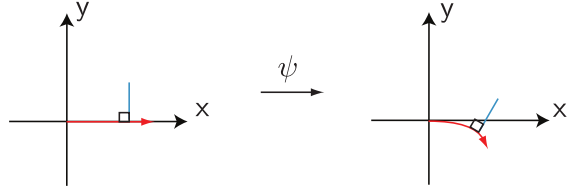}
\end{center}
\caption{The unstable manifold in coordinates $\mathbf{q=}(\mathsf{x}%
,\mathsf{y}_{1},\mathsf{y}_{2},\mathsf{y}_{3})$ (left), and in coordinates
$\mathbf{v}$ (right).}%
\label{fig:psi-coord}%
\end{figure}

Combining the linear and nonlinear changes of coordinates gives the total
change $\Phi:$ $\mathbb{R}^{4}\rightarrow\mathbb{R}^{4}$ from coordinates
$\mathbf{q}=\left(  \mathsf{x},\mathsf{y}_{1},\mathsf{y}_{2},\mathsf{y}%
_{3}\right)  ,$ defined as%
\begin{equation}
\mathbf{x}=\Phi^{\mu}(\mathbf{q}):=C^{\mu}\psi\left(  \mathbf{q}\right)  .
\label{eq:total-change}%
\end{equation}
The vector field in coordinates $\mathbf{q}$ is%
\begin{equation}
\hat{F}(\mathbf{q})=D\left(  \Phi^{\mu}\right)  ^{-1}\left(  \Phi^{\mu
}(\mathbf{q})\right)  F\left(  \Phi^{\mu}(\mathbf{q})\right)  .
\label{eq:Flocal}%
\end{equation}

\begin{remark}
In our application, the nonlinear change of coordinates is not strictly
necessary. Even without it we can obtain our result, but with smaller
accuracy. We decided to add the nonlinear change in order to demonstrate that
such techniques are possible. Also, with a nonlinear change of coordinates
some more careful consideration is needed when computing the derivative of the
vector field in local coordinates. This is discussed in section
\ref{sec:manifold-3bp}.
\end{remark}

\subsection{Enclosure of the unstable manifold\label{sec:manifold-3bp}}

In order to obtain an enclosure of the unstable manifold in coordinates
$\mathbf{q}$ we apply Theorem \ref{th:odes-wu} to establish the existence of
the manifold.

Let us first specify our change of coordinates $\Phi^{\mu}$ (see
(\ref{eq:total-change})). The linear part $C^{\mu}$ of $\Phi^{\mu}$ is given
by (\ref{eq:C-mu}). We consider an interval of parameters%
\[
\boldsymbol{\mu}=0.004253863522+10^{-10}\left[  -1,1\right]  ,
\]
and for any $\mu\in\boldsymbol{\mu}$ take the same nonlinear change $\psi$
(see (\ref{eq:psi-form})), with $K$ chosen as
\begin{align}
K_{0}\left(  \mathsf{x}\right)   &  =\mathsf{x},\label{eq:K-cooef}\\
K_{1}\left(  \mathsf{x}\right)   &  =-0.4426997319120566\mathsf{x}%
^{2}+0.2117307906593041\mathsf{x}^{3},\nonumber\\
K_{2}\left(  \mathsf{x}\right)   &  =0.7204702544171099\mathsf{x}%
^{2}-0.2077414984788253\mathsf{x}^{3},\nonumber\\
K_{3}\left(  \mathsf{x}\right)   &  =0.6096754412253178\mathsf{x}%
^{2}-1.6248371332133488\mathsf{x}^{3}.\nonumber
\end{align}

The first step is to obtain an enclosure of the fixed point in local
coordinates $\mathbf{q}$ (see (\ref{eq:total-change})). We do this by applying
the interval Newton method (Theorem \ref{th:interval-Newton}). In order to do
this we have to compute the derivative of the local vector field
(\ref{eq:Flocal}) as follows. Since%
\[
D\left(  \left(  \Phi^{\mu}\right)  ^{-1} \right)  \left(  \Phi^{\mu}\left(
\mathbf{q}\right)  \right)  =\left(  D\Phi^{\mu}\left(  \mathbf{q}\right)
\right)  ^{-1},
\]
we see that%
\[
D\Phi^{\mu}\left(  \mathbf{q}\right)  \hat{F}(\mathbf{q})=F\left(  \Phi^{\mu
}(\mathbf{q})\right)  .
\]
Differentiating on both sides gives%
\[
D^{2}\Phi^{\mu}\left(  \mathbf{q}\right)  \hat{F}(\mathbf{q})+D\Phi^{\mu
}\left(  \mathbf{q}\right)  D\hat{F}(\mathbf{q})=DF\left(  \Phi^{\mu
}(\mathbf{q})\right)  D\Phi^{\mu}(\mathbf{q}),
\]
hence%
\begin{equation}
D\hat{F}(\mathbf{q})=\left(  D\Phi^{\mu}\left(  \mathbf{q}\right)  \right)
^{-1}\left(  DF\left(  \Phi^{\mu}(\mathbf{q})\right)  D\Phi^{\mu}%
(\mathbf{q})-D^{2}\Phi^{\mu}\left(  \mathbf{q}\right)  \hat{F}(\mathbf{q}%
)\right)  . \label{eq:DFlocal}%
\end{equation}
The main advantage of this approach is that we do not need to invert
$\Phi^{\mu}$ to apply (\ref{eq:DFlocal}). Using (\ref{eq:DFlocal}) and Theorem
\ref{th:interval-Newton} we can establish that for all $\mu\in\boldsymbol{\mu
}$ the fixed point is in a set which we denote as $B$.

The second step is to verify assumptions of Theorem \ref{th:main-unstable}
using Lemma \ref{lem:cc-flows}. In order to do so we choose $\alpha_{h}%
,\alpha_{v}$, and take%
\begin{equation}
N=B+\left[  0,r_{u}\right]  \times\left[  -r_{u}\sqrt{\alpha_{h}},r_{u}%
\sqrt{\alpha_{h}}\right]  ^{3}. \label{eq:N-set}%
\end{equation}
To obtain a good enclosure we subdivide the set $N$, and compute the
derivative on smaller subsets. Numerical results are listed in section
\ref{sec:comp-assisted}. The unstable manifold expressed in local coordinates
$\mathbf{q}$ passes through (see Figure \ref{fig:u3bp})
\begin{equation}
U=B+\{\sqrt{1-\alpha_{v}}\}\times\left[  -\sqrt{\alpha_{h}},\sqrt{\alpha_{h}%
}\right]  ^{3}. \label{eq:U-set}%
\end{equation}

\begin{figure}[ptb]
\begin{center}
\includegraphics[height=3.5cm]{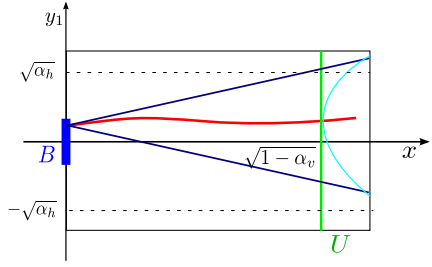}
\end{center}
\caption{The unstable manifold expressed in local coordinates (in red) passes
through a box $U$ (in green). The enclosure $B$ of a fixed point is a small
blue box around $0$.}%
\label{fig:u3bp}%
\end{figure}

\subsection{Proof of existence of a homoclinic connection}

Let $U_{\mu}$ be a set which contains a point on the unstable manifold of
$L_{1}^{\mu}.$ Assume that for any $\mu\in\boldsymbol{\mu}$ the Poincar\'{e}
map%
\begin{align}
P_{\mu} &  :U_{\mu}\rightarrow\{y=0\},\nonumber\\
P_{\mu}(\mathbf{x}) &  :=\phi_{\tau(\mathbf{x})}(\mathbf{x}%
),\label{eq:Pmu-def}%
\end{align}
where%
\[
\tau\left(  \mathbf{x}\right)  =\inf\left\{  t|t>0,\phi_{t}(\mathbf{x}%
)\in\{y=0\}\right\}  ,
\]
is well defined.

\begin{lemma}
Assume that $\boldsymbol{\mu}=\left[  \mu_{\mathrm{left}},\mu_{\mathrm{right}%
}\right]  .$If
\begin{align}
\pi_{p_{x}}P_{\mu_{\mathrm{left}}}(\mathbf{x})  &  <0,\quad\text{for any
}\mathbf{x}\in U_{\mu_{\mathrm{left}}},\label{eq:mu-cond-1}\\
\pi_{p_{x}}P_{\mu_{\mathrm{right}}}(\mathbf{x})  &  >0,\quad\text{for any
}\mathbf{x}\in U_{\mu_{\mathrm{right}}}, \label{eq:mu-cond-2}%
\end{align}
then there exists a $\mu\in\left(  \mu_{\mathrm{left}},\mu_{\mathrm{right}%
}\right)  $, for which we have a homoclinic orbit to $L_{1}^{\mu}.$
\end{lemma}

\begin{proof}
Let $\mathbf{x}_{\mu}$ be any point from the intersection of $U_{\mu}$ with
$W^{u}(L_{1}^{\mu}).$ If
\begin{equation}
P_{\mu}(\mathbf{x}_{\mu})=\left(  x,0,0,p_{y}\right)  , \label{eq:Pmu-form}%
\end{equation}
for some $x,p_{y}$ (which depend on $\mu$), then the point $P_{\mu}%
(\mathbf{x}_{\mu})$ is $S$-symmetric (see (\ref{eq:S-sym})), and by
(\ref{eq:S-sym-property})
\[
S(\phi_{t}(P_{\mu}(\mathbf{x}_{\mu})))=\phi_{-t}(S(P_{\mu}(\mathbf{x}_{\mu
})))=\phi_{-t}(P_{\mu}(\mathbf{x}_{\mu}))\overset{t\rightarrow+\infty
}{\rightarrow}L_{1}^{\mu}.
\]
This means that $P_{\mu}(\mathbf{x}_{\mu})$ lies on a homoclinic orbit to
$L_{1}^{\mu}.$

We need to prove that there exists a parameter $\mu\in\boldsymbol{\mu},$ for
which $P_{\mu}(\mathbf{x}_{\mu})$ would be of the form (\ref{eq:Pmu-form}). By
definition of $P_{\mu}$ (\ref{eq:Pmu-def}), we know that $\pi_{y}P_{\mu
}(\mathbf{x}_{\mu})=0.$ It is therefore sufficient to show that for some
$\mu\in\left(  \mu_{\mathrm{left}},\mu_{\mathrm{right}}\right)  $
\[
\pi_{p_{x}}P_{\mu}(\mathbf{x}_{\mu})=0.
\]
Let
\[
g:\left[  \mu_{\mathrm{left}},\mu_{\mathrm{right}}\right]  \rightarrow
\mathbb{R}%
\]
be defined as%
\[
g(\mu)=\pi_{p_{x}}P_{\mu}(\mathbf{x}_{\mu}).
\]
By (\ref{eq:mu-cond-1}--\ref{eq:mu-cond-2}) we see that $g(\mu_{\mathrm{left}%
})<0<g(\mu_{\mathrm{right}})$. By continuity of the flow with respect to the
parameters of the vector field, we know that $g$ is continuous, hence
existence of $\mu$ for which $g(\mu)=0$ follows from the Bolzano theorem.
\end{proof}

For a given $\mu\in\boldsymbol{\mu}$ we can obtain the enclosure $U_{\mu}$
using the method described in section \ref{sec:manifold-3bp}. In fact, the
method can be applied not only for a single parameter $\mu\in\boldsymbol{\mu}%
$, but for an interval of parameters. Conditions (\ref{eq:mu-cond-1}%
--\ref{eq:mu-cond-2}) can be verified by integrating the system numerically,
using a rigorous, interval arithmetic based integrator. Such tool is available
as a part of the CAPD\footnote{Computer Assisted Proofs in Dynamics
http://capd.ii.uj.edu.pl} library. The package can compute Poincar\'{e} maps
$P_{\mu}$ on prescribed parameter intervals. As the Poincar\'{e} map is
computed, at the same time it is verified that it is well defined.

\subsection{Computer assisted bounds\label{sec:comp-assisted}}

Let us first take
\[
\mu_{\mathrm{left}}=0.004253863522-10^{-10}.
\]
In local coordinates $\mathbf{q}$, the enclosure $B$ for the fixed point is%
\[
B=10^{-15}\left(
\begin{array}
[c]{c}%
\lbrack-1.137,1.169]\\
\lbrack-0.426,0.394]\\
\lbrack-0.181,0.181]\\
\lbrack-0.180,0.308]
\end{array}
\right)  .
\]
For the enclosure $N$ of the unstable manifold (\ref{eq:N-set}) in coordinates
$\mathbf{q}$ we take
\[
\alpha_{h}=10^{-8},\qquad\alpha_{v}=10^{-4},\qquad r_{u}=10^{-7}.
\]
The enclosure (displayed with rough rounding, which ensures true enclosure) of
the derivative of the vector field in local coordinates is%
\begin{align*}
&  \left[  D\hat{F}(N)\right]  =\\
&  \left(
\begin{array}
[c]{cccc}%
\text{{\footnotesize [2.80038,2.80039]}} & \text{{\footnotesize 10}%
}^{\text{{\footnotesize -6}}}\text{{\footnotesize [-0.0065,1.281]}} &
\text{{\footnotesize 10}}^{\text{{\footnotesize -9}}}%
\text{{\footnotesize [-1.469,1.468]}} & \text{{\footnotesize 10}%
}^{\text{{\footnotesize -7}}}\text{{\footnotesize [-6.752,0.032]}}\\
\text{{\footnotesize 10}}^{\text{{\footnotesize -9}}}%
\text{{\footnotesize 8.521[-1,1]}} & \text{{\footnotesize [-2.80039,-2.80038]}%
} & \text{{\footnotesize 10}}^{\text{{\footnotesize -6}}}%
\text{{\footnotesize [-0.0015,1.01]}} & \text{{\footnotesize 10}%
}^{\text{{\footnotesize -7}}}\text{{\footnotesize [-5.352,0.032]}}\\
\text{{\footnotesize 10}}^{\text{{\footnotesize -9}}}%
\text{{\footnotesize 6.035[-1,1]}} & \text{{\footnotesize 10}}%
^{\text{{\footnotesize -7}}}\text{{\footnotesize [-6.752,0.0320]}} &
\text{{\footnotesize 10}}^{\text{{\footnotesize -7}}}%
\text{{\footnotesize [-2.659,0.0044]}} &
\text{{\footnotesize [2.25179,2.25180]}}\\
\text{{\footnotesize 10}}^{\text{{\footnotesize -9}}}%
\text{{\footnotesize 4.053[-1,1]}} & \text{{\footnotesize 10}}%
^{\text{{\footnotesize -9}}}\text{{\footnotesize [-1.468,1.469]}} &
\text{{\footnotesize [-2.25180,-2.25179]}} & \text{{\footnotesize 10}%
}^{\text{{\footnotesize -7}}}\text{{\footnotesize [-0.0044,2.66]}}%
\end{array}
\right)
\end{align*}
To apply Theorem \ref{th:odes-wu} we take $c_{v}=2.8$ (which looking at $[
D\hat{F}(N)]$ is clearly close to its unstable eigenvalue) and $c_{h}=1$ (here
we arbitrarily chose a number from $(0,c_{v})$), and verify conditions
(\ref{eq:ode-ass1}--\ref{eq:ode-ass4}).

The set $U$ defined in (\ref{eq:U-set}), when transported to the original
coordinates is equal to (displayed with rough rounding, which ensures true
enclosure)%
\[
U_{\mu_{\mathrm{left}}}=L_{1}^{\mu_{\mathrm{left}}}+10^{-8}\left(
\begin{array}
[c]{c}%
\lbrack4.007,4.008]\\
\lbrack-1.934,-1.931]\\
\lbrack13.15,13.16]\\
\lbrack-1.407,-1.402]
\end{array}
\right)  .
\]
This shows that we enclose the unstable manifold very close to the fixed point.

After propagating the set $U_{\mu_{_{\mathrm{left}}}}$ to the section
$\{y=0\}$ we obtain an estimate on the image by the Poincar\'{e} map%
\[
\left[  P_{\mu_{_{\mathrm{left}}}}\left(  U_{\mu_{_{\mathrm{left}}}}\right)
\right]  =\left(
\begin{array}
[c]{c}%
0.8270258829+10^{-10}\left[  -1,1\right]  \\
0\\
-10^{-8}[7.501,2.915]\\
0.9251225636+10^{-10}\left[  -1,1\right]
\end{array}
\right)  .
\]
The important result is that on the third coordinate we have values smaller
than zero, which verifies (\ref{eq:mu-cond-1}).

For
\[
\mu_{\mathrm{right}}=0.004253863522+10^{-10},
\]
up to the rounding used to present the result in this paper, the estimates on
$B$ and $[D\hat{F}(N)]$ are indistinguishable from the ones for $\mu
_{\mathrm{left}}$. The important fact though is that%
\[
\left[  P_{\mu_{_{\mathrm{right}}}}\left(  U_{\mu_{_{\mathrm{right}}}}\right)
\right]  =\left(
\begin{array}
[c]{c}%
0.82702588075+10^{-10}\left[  -1,1\right] \\
0\\
10^{-8}[2.825,7.421],\\
0.9251225623+10^{-10}\left[  -1,1\right]
\end{array}
\right)  ,
\]
is positive on the third coordinate, which ensures (\ref{eq:mu-cond-2}).

To verify that $P_{\mu}$ is well defined for all $\mu\in\boldsymbol{\mu}$,
similar computations, but with lesser accuracy, were performed.

The computer assisted proof takes 4.27 seconds, on a single core Intel i7
processor, with 1.90GHz. Majority of this time was spent on verifying that
$P_{\mu}$ is well defined for all $\mu\in\boldsymbol{\mu}$. In order to do so,
the parameter interval was subdivided into 20 fragments $\boldsymbol{\mu
}=\boldsymbol{\mu}_{1}\cup\ldots,\cup\boldsymbol{\mu}_{20}$, and each time we
needed to integrate from $U_{\boldsymbol{\mu}_{i}}$ to the section $\{y=0\},$
for $i=1,\ldots,20,$ which was time consuming.


\section{Closing remarks\label{sec:closing-remarks}}

The paper presents a new method for establishing of strong (un)stable manifolds for fixed points. The method can be applied for computer assisted
proofs. We have shown an application of our method in the context of the
planar circular restricted three body problem, proving that there exists a
homoclinic orbit to the libration point $L_{1}$ for a suitably chosen mass
parameter. Our method produced a tight enclosure of the manifold and also a
tight enclosure for the mass parameter for which the manifold leads to a
homoclinic connection.

\section{Acknowledgements\label{sec:acknowledgements}}

We would like to thank Piotr Zgliczy\'nski for reading our work and for very helpful comments and corrections. We also thank Rafael de la Llave for discussions concerning
the application of the parameterization method. We thank
him also for his \texttt{C} code, that we have used to obtain the suitable change of
coordinates in the application of our method for the restricted three body
problem. All the computer assisted proofs presented in this paper have been
performed using the \texttt{CAPD}\footnote{http://capd.ii.uj.edu.pl} library.
We thank Tomasz Kapela and Daniel Wilczak for
discussions concerning the use of the library. Finally, we thank the anonymous referee for very helpful comments, which have enabled us to improve the paper.


\section{Appendix\label{sec:appendix}}

\subsection{Proof of Lemma \ref{lem:cc-pos-def}\label{app:cc-pos-def}}

\begin{proof}
For any $q^{\ast},q^{\ast\ast}\in N$%
\[
\frac{d}{dt}f(\left(  1-t\right)  q^{\ast\ast}+tq^{\ast})=Df\left(  \left(
1-t\right)  q^{\ast\ast}+tq^{\ast}\right)  \left(  q^{\ast}-q^{\ast\ast
}\right)  ,
\]
hence%
\[
f(q^{\ast})-f\left(  q^{\ast\ast}\right)  =\int_{0}^{1}Df\left(  \left(
1-t\right)  q^{\ast\ast}+tq^{\ast}\right)  dt\left(  q^{\ast}-q^{\ast\ast
}\right)  .
\]
Since%
\[
B=\int_{0}^{1}Df\left(  \left(  1-t\right)  q^{\ast\ast}+tq^{\ast}\right)
dt\in\left[  Df\left(  N\right)  \right]  ,
\]
we see that for $q^{\ast}\neq q^{\ast\ast}$
\[
Q\left(  f(q^{\ast})-f\left(  q^{\ast\ast}\right)  \right)  -mQ\left(
q^{\ast}-q^{\ast\ast}\right)  =Q\left(  B\left(  q^{\ast}-q^{\ast\ast}\right)
\right)  -mQ\left(  q^{\ast}-q^{\ast\ast}\right)  >0,
\]
as required.
\end{proof}

\subsection{Proof of Remark \ref{rem:cones-setup}\label{app:cones-setup}}

\begin{proof}
Here we prove point \ref{rem:cones-setup1} of Remark \ref{rem:cones-setup}. If
$\left\Vert x\right\Vert \leq1$ and $Q_{h}\left(  x,y\right)  \geq\alpha
_{h}-1$, then
\[
\alpha_{h}-\left\Vert y\right\Vert ^{2}\geq\alpha_{h}\left\Vert x\right\Vert
^{2}-\left\Vert y\right\Vert ^{2}=Q_{h}\left(  x,y\right)  \geq\alpha_{h}-1,
\]
hence $\left\Vert y\right\Vert \leq1.$

Now we prove point \ref{rem:cones-setup2} of Remark \ref{rem:cones-setup}. If
$\left\Vert y\right\Vert \leq1$ and $Q_{v}\left(  x,y\right)  \leq1-\alpha
_{v}$, then
\[
\left\Vert x\right\Vert ^{2}-\alpha_{v}\leq\left\Vert x\right\Vert ^{2}%
-\alpha_{v}\left\Vert y\right\Vert ^{2}=Q_{v}\left(  x,y\right)  \leq
1-\alpha_{v},
\]
hence $\left\Vert x\right\Vert \leq1.$

To prove point \ref{rem:cones-setup2a} of Remark \ref{rem:cones-setup}, assume
that $Q_{h}\left(  x,y\right)  \geq\alpha_{h}-1$ and $Q_{v}\left(  x,y\right)
\leq1-\alpha_{v}$. If $||x||\leq1$ or $||y||\leq1$ then from points
\ref{rem:cones-setup1}. and \ref{rem:cones-setup2}. of the remark, we get
$(x,y)\in N$. Suppose that $||x||>1$ and $||y||>1$. From the assumptions,
\[
\left\{
\begin{array}
[c]{l}%
\alpha_{h}\left\Vert x\right\Vert ^{2}-\left\Vert y\right\Vert ^{2}\geq
\alpha_{h}-1,\\
\left\Vert x\right\Vert ^{2}-\alpha_{v}\left\Vert y\right\Vert ^{2}%
\leq1-\alpha_{v}.
\end{array}
\right.
\]
Thus, rearranging the above inequalities gives
\[
\left\{
\begin{array}
[c]{l}%
\frac{||x||^{2}-1}{||y||^{2}-1}\geq\frac{1}{\alpha_{h}}>1,\\
\frac{||x||^{2}-1}{||y||^{2}-1}\leq\alpha_{v}<1,
\end{array}
\right.
\]
which is a contradiction. This implies that $(x,y)\in N$.

Here we prove point \ref{rem:cones-setup3} of Remark \ref{rem:cones-setup}. If
$\left\Vert y\right\Vert \leq a,$ then%
\[
Q_{h}\left(  x,y\right)  =\alpha_{h}\left\Vert x\right\Vert ^{2}-\left\Vert
y\right\Vert ^{2}\geq\alpha_{h}\left\Vert x\right\Vert ^{2}-a^{2}\geq-a^{2},
\]
as required.
\end{proof}

\subsection{Proof of Lemma \ref{lem:positive-cone-back-conv}%
\label{app:positive-cone-back-conv}}

\begin{proof}
Let us write $q_{k}=\left(  x_{k},y_{k}\right)  $. Since $Q_{h}(x_{k}%
,y_{k})\geq0,$%
\begin{equation}
\left\Vert x_{k}\right\Vert ^{2}\geq\alpha_{h}\left\Vert x_{k}\right\Vert
^{2}\geq\left\Vert y_{k}\right\Vert ^{2}. \label{eq:tmp-Qh-slope}%
\end{equation}

Let us observe that\textbf{ }%
\[
Q_{v}(x_{k},y_{k})=\left\Vert x_{k}\right\Vert ^{2}-\alpha_{v}\left\Vert
y_{k}\right\Vert ^{2}\geq\alpha_{h}\left\Vert x_{k}\right\Vert ^{2}-\left\Vert
y_{k}\right\Vert ^{2}=Q_{h}(x_{k},y_{k})\geq0,
\]
which implies that $q_{k}\in\left\{  Q_{v}\geq0\right\}  $. From $(Q_{v}%
,m_{v})$ cone conditions follows that for $k\leq0,$
\[
Q_{v}\left(  q_{k}\right)  =Q_{v}\left(  f\left(  q_{k-1}\right)
-f(0)\right)  \geq m_{v}Q_{v}(q_{k-1}-0)=m_{v}Q_{v}(q_{k-1})\geq0.
\]
This implies that for any $k\leq0$,
\begin{equation}
Q_{v}\left(  q_{0}\right)  \geq m_{v}^{|k|}Q_{v}(q_{k}).
\label{eq:Qv-contraction-tmp}%
\end{equation}
Since $\alpha_{h},\alpha_{v}\in\left(  0,1\right)  $, by
(\ref{eq:tmp-Qh-slope}),%
\[
Q_{v}(q_{k})=\left\Vert x_{k}\right\Vert ^{2}-\alpha_{v}\left\Vert
y_{k}\right\Vert ^{2}\geq\left(  1-\alpha_{v}\alpha_{h}\right)  \left\Vert
x_{k}\right\Vert ^{2}\geq\left(  1-\alpha_{v}\alpha_{h}\right)  \left\Vert
y_{k}\right\Vert ^{2},
\]
hence from (\ref{eq:Qv-contraction-tmp}), for any $k\leq0$,
\begin{equation}
\left\Vert x_{k}\right\Vert ^{2}+\left\Vert y_{k}\right\Vert ^{2}\leq2\left(
1-\alpha_{v}\alpha_{h}\right)  ^{-1}Q_{v}(q_{k})\leq2\left(  1-\alpha
_{v}\alpha_{h}\right)  ^{-1}m_{v}^{k}Q_{v}\left(  q_{0}\right)  .
\label{eq:expansion-total-est}%
\end{equation}
Since $Q_{v}\left(  q_{0}\right)  =\left\Vert x_{0}\right\Vert ^{2}-\alpha
_{v}\left\Vert y_{0}\right\Vert ^{2}\leq\left\Vert x_{0}\right\Vert ^{2}%
\leq1,$ (\ref{eq:expansion-total-est}) gives%
\[
\left\Vert q_{k}\right\Vert =\sqrt{\left\Vert x_{k}\right\Vert ^{2}+\left\Vert
y_{k}\right\Vert ^{2}}\leq\sqrt{2\left(  1-\alpha_{v}\alpha_{h}\right)  ^{-1}%
}\sqrt{m_{v}}^{k},
\]
as required.
\end{proof}

\subsection{Proof of Lemma \ref{lem:negative-cone-forward-conv}%
\label{app:negative-cone-forward-conv}}

\begin{proof}
The proof follows along the same lines as the proof of Lemma
\ref{lem:positive-cone-back-conv}.

Let us use the notation $\left(  x_{k},y_{k}\right)  =f^{k}(q_{0})$. Since
$\left(  x_{k},y_{k}\right)  \in\{Q_{v}\leq0\}$%
\begin{equation}
\left\Vert x_{k}\right\Vert ^{2}\leq\alpha_{v}\left\Vert y_{k}\right\Vert
^{2}\leq\left\Vert y_{k}\right\Vert ^{2}. \label{eq:tmp-cone-contraction}%
\end{equation}

Let us observe that
\begin{multline*}
Q_{h}(x,y)=\alpha_{h}\left\Vert x\right\Vert ^{2}-\left\Vert y\right\Vert
^{2}\leq\alpha_{h}\left(  \left\Vert x\right\Vert ^{2}-\left\Vert y\right\Vert
^{2}\right) \\
\leq\alpha_{h}\left(  \left\Vert x\right\Vert ^{2}-\alpha_{v}\left\Vert
y\right\Vert ^{2}\right)  =\alpha_{h}Q_{v}(x,y)\leq0,
\end{multline*}
which implies that $f^{k}(q_{0})\in\{Q_{h}\leq0\}$. From $(Q_{h},m_{h})$ cone
conditions follows that for $k\geq0,$%
\begin{equation}
0\geq Q_{h}\left(  f^{k}(q_{0})\right)  \geq m_{h}^{k}Q_{h}\left(
q_{0}\right)  . \label{eq:tmp-cone-contraction-2}%
\end{equation}
Since $\alpha_{h},\alpha_{v}\in\left(  0,1\right)  $, by
(\ref{eq:tmp-cone-contraction}) and (\ref{eq:tmp-cone-contraction-2}),%
\begin{align}
\left(  1-\alpha_{h}\alpha_{v}\right)  \left\Vert x_{k}\right\Vert ^{2}  &
\leq\left(  1-\alpha_{h}\alpha_{v}\right)  \left\Vert y_{k}\right\Vert
^{2}\label{eq:tmp-cone-contraction-3}\\
&  \leq\left\Vert y_{k}\right\Vert ^{2}-\alpha_{h}\left\Vert x_{k}\right\Vert
^{2}=-Q_{h}(f^{k}(q_{0}))\leq m_{h}^{k}\left\vert Q_{h}\left(  q_{0}\right)
\right\vert .\nonumber
\end{align}
Since
\begin{equation}
\left\vert Q_{h}\left(  q_{0}\right)  \right\vert =\left\Vert y_{0}\right\Vert
^{2}-\alpha_{h}\left\Vert x_{0}\right\Vert ^{2}\leq\left\Vert y_{0}\right\Vert
^{2}\leq1, \label{eq:tmp-cone-contraction-4}%
\end{equation}
combining (\ref{eq:tmp-cone-contraction-3}) and
(\ref{eq:tmp-cone-contraction-4}),%
\[
\left\Vert f^{k}(q_{0})\right\Vert ^{2}=\left\Vert x_{k}\right\Vert
^{2}+\left\Vert y_{k}\right\Vert ^{2}\leq2\left(  1-\alpha_{h}\alpha
_{v}\right)  ^{-1}m_{h}^{k}\left\vert Q_{h}\left(  q_{0}\right)  \right\vert
\leq2\left(  1-\alpha_{h}\alpha_{v}\right)  ^{-1}m_{h}^{k},
\]
as required.
\end{proof}

\subsection{Proof of Lemma \ref{lem:gronwell-est}\label{app:gronwell-est}}

The proof of Lemma \ref{lem:gronwell-est} is based on the Gronwall lemma. We
start by writing out its statement.

\begin{lemma}
\label{lem:gronwall} \cite{CL} (Gronwall lemma) If $u,v,c\geq0$ on $[0,t]$,
$c$ is differentiable, and%
\[
v\left(  t\right)  \leq c(t)+\int_{0}^{t}u(s)v(s)ds
\]
then%
\[
v(t)\leq c(0)\exp\left(  \int_{0}^{t}u(s)ds\right)  +\int_{0}^{t}c^{\prime
}(s)\left[  \exp\left(  \int_{s}^{t}u\left(  \tau\right)  d\tau\right)
\right]  ds.
\]

\end{lemma}

We are now ready to give the proof:

\begin{proof}
[Proof of Lemma \ref{lem:gronwell-est}]We start by proving
(\ref{eq:gronw-est-1}) for $t>0$. Let us fix $p_{1}\neq p_{2}$ and consider
$v(t)=\left\Vert g_{1}\left(  p_{1},p_{2},t\right)  \right\Vert .$ Since
$g_{1}\left(  p_{1},p_{2},0\right)  =0,$
\begin{align*}
v\left(  t\right)   &  =\left\Vert \int_{0}^{t}\frac{d}{ds}g_{1}\left(
p_{1},p_{2},s\right)  ds\right\Vert \\
&  =\left\Vert \int_{0}^{t}F\left(  \phi_{s}(p_{1})\right)  -F\left(  \phi
_{s}(p_{2})\right)  ds\right\Vert \\
&  \leq\int_{0}^{t}L\left\Vert \phi_{s}(p_{1})-\phi_{s}(p_{2})-\left(
p_{1}-p_{2}\right)  \right\Vert ds+\int_{0}^{t}L\left\Vert p_{1}%
-p_{2}\right\Vert dt.
\end{align*}
Taking $c(t)=tL\left\Vert p_{1}-p_{2}\right\Vert $ and $u(t)=L$, by Lemma
\ref{lem:gronwall},%
\begin{align*}
v(t) &  \leq\int_{0}^{t}L\left\Vert p_{1}-p_{2}\right\Vert \left[  \exp\left(
\int_{s}^{t}Ld\tau\right)  \right]  ds\\
&  =L\left\Vert p_{1}-p_{2}\right\Vert \frac{1}{L}\left(  e^{tL}-1\right)  ,
\end{align*}
which concludes the proof of (\ref{eq:gronw-est-1}) for $t>0$. For negative
times, the proof follows by taking $v(t)=g\left(  p_{1},p_{2},-t\right)  $,
with $t>0$, and performing mirror computations.

We now prove (\ref{eq:gronw-est-2}) for $t>0$. We first observe that by our
assumptions (\ref{eq:vect-fld-as-1}) and (\ref{eq:vect-fld-as-2}) on the
vector field $F$ follows that for $s>0$,%
\begin{align}
\left\Vert \phi_{s}(p_{1})-\phi_{s}(p_{2})\right\Vert  &  \leq\left\Vert
p_{1}-p_{2}\right\Vert e^{Ls},\label{eq:gronw-tmp-1}\\
\left\Vert DF\left(  \phi_{s}(p_{1})\right)  -DF\left(  \phi_{s}%
(p_{2})\right)  \right\Vert  &  \leq M\left\Vert \phi_{s}(p_{1})-\phi
_{s}(p_{2})\right\Vert ,\label{eq:gronw-tmp-2}\\
\left\Vert \left[  DF\left(  \phi_{s}(p_{1})\right)  -DF\left(  \phi_{s}%
(p_{2})\right)  \right]  F(\phi_{s}(p_{2}))\right\Vert  &  \leq\mu\left\Vert
DF\left(  \phi_{s}(p_{1})\right)  -DF\left(  \phi_{s}(p_{2})\right)
\right\Vert .\label{eq:gronw-tmp-3}%
\end{align}
We take $v(t)=\left\Vert g_{2}\left(  p_{1},p_{2},t\right)  \right\Vert $, and
compute, (using (\ref{eq:gronw-tmp-1}--\ref{eq:gronw-tmp-3}) in the second
inequality,)
\begin{align*}
v\left(  t\right)   &  =\left\Vert \int_{0}^{t}\frac{d}{ds}\left(  F(\phi
_{s}(p_{1}))-F(\phi_{s}(p_{2}))\right)  -\left(  F(p_{1})-F(p_{2})\right)
ds\right\Vert \\
&  =\left\Vert \int_{0}^{t}DF\left(  \phi_{s}(p_{1})\right)  F\left(  \phi
_{s}(p_{1})\right)  -DF\left(  \phi_{s}(p_{2})\right)  F\left(  \phi_{s}%
(p_{2})\right)  ds\right\Vert \\
&  \leq\int_{0}^{t}\left\Vert DF\left(  \phi_{s}(p_{1})\right)  F\left(
\phi_{s}(p_{1})\right)  -DF\left(  \phi_{s}(p_{1})\right)  F\left(  \phi
_{s}(p_{2})\right)  \right\Vert ds\\
&  \quad+\int_{0}^{t}\left\Vert DF\left(  \phi_{s}(p_{1})\right)  F\left(
\phi_{s}(p_{2})\right)  -DF\left(  \phi_{s}(p_{2})\right)  F\left(  \phi
_{s}(p_{2})\right)  \right\Vert ds\\
&  \leq\int_{0}^{t}L\left\Vert F\left(  \phi_{s}(p_{1})\right)  -F\left(
\phi_{s}(p_{2})\right)  \right\Vert ds+\int_{0}^{t}\mu Me^{Ls}\left\Vert
p_{1}-p_{2}\right\Vert ds\\
&  \leq\int_{0}^{t}Lv(s)ds+\int_{0}^{t}L\left\Vert F(p_{1})-F(p_{2}%
)\right\Vert ds+\int_{0}^{t}\mu Me^{Ls}\left\Vert p_{1}-p_{2}\right\Vert ds\\
&  \leq\int_{0}^{t}Lv(s)ds+\int_{0}^{t}\left(  L^{2}+\mu Me^{Ls}\right)
\left\Vert p_{1}-p_{2}\right\Vert ds.
\end{align*}
Taking $c(t)=\int_{0}^{t}\left(  L^{2}+\mu Me^{Ls}\right)  \left\Vert
p_{1}-p_{2}\right\Vert ds$ and $u(t)=L,$ by Lemma \ref{lem:gronwall},%
\begin{align*}
v(t) &  \leq\int_{0}^{t}\left(  L^{2}+\mu Me^{Ls}\right)  \left\Vert
p_{1}-p_{2}\right\Vert \left[  \exp\left(  \int_{s}^{t}Ld\tau\right)  \right]
ds\\
&  =\left(  L\left(  e^{Lt}-1\right)  +te^{Lt}\mu M\right)  \left\Vert
p_{1}-p_{2}\right\Vert .
\end{align*}
This concludes the proof of (\ref{eq:gronw-est-2}) for $t>0$. For negative
times, we take $v(t)=\left\Vert g(p_{1},p_{2},-t)\right\Vert $, with $t>0$,
and perform mirror computations.
\end{proof}

\end{document}